\documentclass[a4paper,11pt]{article}%
\usepackage{amssymb}
\usepackage{amsmath}
\usepackage{amsfonts}
\usepackage{graphicx}
\usepackage{boxedminipage}%
\providecommand{\U}[1]{\protect\rule{.1in}{.1in}}
\newtheorem{theorem}{Theorem}

\newtheorem{corollary}[theorem]{Corollary}

\newtheorem{example}[theorem]{Example}

\newtheorem{remark}[theorem]{Remark}

\newenvironment{proof}[1][Proof]{\noindent\textbf{#1.} }{\ \rule{0.5em}{0.5em}}
\begin{document}

\title{An Alternative Approach to Elliptical Motion}
\author{Mustafa \"{O}zdemir\thanks{Department of Mathematics, Akdeniz University,
Antalya, TURKEY, e-mail: mozdemir@akdeniz.edu.tr}}
\maketitle

\begin{abstract}
Elliptical rotation is the motion of a point on an ellipse through some angle
about a vector. The purpose of this paper is to examine the generation of
elliptical rotations and to interpret the motion of a point on an elipsoid
using elliptic inner product and elliptic vector product. To generate an
elliptical rotation matrix, first we define an elliptical ortogonal matrix and
an elliptical skew symmetric matrix using the associated inner product. Then
we use elliptic versions of the famous Rodrigues, Cayley, and Householder
methods to construct an elliptical rotation matrix. Finally, we define
elliptic quaternions and generate an elliptical rotation matrix using those
quaternions. Each method is proven and is provided with several numerical examples.



\end{abstract}

\section{Introduction}

\qquad A rotation is an example of an isometry, a map that moves points
without changing the distances between them. A rotation is a linear
transformation that describes the motion of a rigid body around a fixed point
or an axis and can be expressed with an orthonormal matrix which is called a
rotation matrix. $n\times n$ rotation matrices form a special orthogonal
group, denoted by $\mathbf{SO}(n)$, which, for $n>2,$ is non-abelian. The
group of $n\times n$ rotation matrices is isomorphic to the group of rotations
in an $n$ dimensional space. This means that multiplication of rotation
matrices corresponds to composition of rotations. Rotation matrices are used
extensively for computations in geometry, kinematics, physics, computer
graphics, animations, and optimization problems involving the estimation of
rigid body transformations. For this reason, the generation of a rotation
matrix is considered to be an important problem in mathematics.

\qquad In the two dimensional Euclidean space, a rotation matrix can easily be
generated using basic linear algebra or complex numbers. Similarly, in the
Lorentzian plane, a rotation matrix can be generated by double (hyperbolic)
numbers. In higher dimensional spaces, obtaining a rotation matrix using the
inner product is impractical since each column and row of a rotation matrix
must be a unit vector perpendicular to all other columns and rows,
respectively. These constraints make it difficult to construct a rotation
matrix using the inner product. Instead, in higher dimensional spaces,
rotation matrices can be generated using various other methods such as unit
quaternions, the Rodrigues formula, the Cayley formula, and the Householder
transformation. We will give a brief review of these methods and use
elliptical versions of these methods later in the paper.

\textbf{1. A unit quaternion : }Each unit quaternion represents a rotation in
the Euclidean 3-space. That is, only four numbers are enough to construct a
rotation matrix, the only constraint being that the norm of the quaternion is
equal to 1. Also, in this method, the rotation angle and the rotation axis can
be determined easily. However, this method is only valid in the three
dimensional spaces (\cite{qua2}, \cite{qua3}). In the Lorentzian space,
timelike split quaternions are used instead of ordinary usual quaternions
(\cite{qua4}, \cite{qua5}).

\textbf{2. Rodrigues Formula : }An orthonormal matrix can be obtained using
the matrix exponential $e^{\theta A}$ where $A$ is a skew symmetric matrix and
$\theta$ is the rotation angle. In this method, only three numbers are needed
to construct a rotation matrix in the Euclidean 3-space (\cite{rod1},
\cite{rod2}, \cite{rod3} and, \cite{rod4}).The vector set up with these three
numbers gives the rotation axis. This method can be extended to the $n$
dimensional Euclidean and Lorentzian spaces (\cite{eberly}, \cite{gallier 2},
\cite{Gallier book}, \cite{MEBIUS2} and, \cite{rot3}).

\textbf{3. Cayley Formula : }The formula $C=\left(  I+A\right)  \left(
I-A\right)  ^{-1}$ gives a rotation matrix, where $A$ is a skew symmetric
matrix. Rotation matrices can be given by the Cayley formula without using
trigonometric functions. The Cayley formula is an easy method but it doesn't
give the rotation angle directly (\cite{melek2}, \cite{cay1}, \cite{cay2},
\cite{cay3} and, \cite{cay4}).

\textbf{4. Householder Transformation : }The Householder transformation gives
us a reflection matrix. We can obtain a rotation matrix using two Householder
transformations. This method is an elegant method but it can be long and
tedious. Also, the rotation angle has to be calculated separately. This
transformation can be used in several scalar product spaces (\cite{Aragon1},
\cite{Aragon2}, \cite{Fuller}, \cite{Macykey1} and, \cite{Uhlig}).

\qquad Details about generating rotation matrices, particularly in the
Euclidean and Lorentzian spaces, using these methods can be found in various
papers, some of which are given in the reference section. Those authors mostly
studied the rotation matrices in the positive definite scalar product space
whose associated matrices are diag$\left(  \pm1,\cdots,\pm1\right)  ,$ and
interpreted the results geometrically. For example, quaternions and timelike
split quaternions were used to generate rotation matrices in the three
dimensional Euclidean and Lorentzian spaces where the associated matrices were
$diag\left(  1,1,1\right)  $ and $diag\left(  -1,1,1\right)  ,$ respectively.
In these spaces, rotations occur on the sphere $x^{2}+y^{2}+z^{2}=r^{2}$ or
the hyperboloids $-x^{2}+y^{2}+z^{2}=\pm r^{2}$. That is, Euclidean and
Lorentzian rotation matrices help us to understand spherical and hyperbolic
rotations. In the Euclidean space, a rotation matrix rotates a point or a
rigid body through a circular angle about an axis. That is, the motion happens
on a circle. Similarly, in the Lorentzian space, a rotation matrix rotates a
point through an angle about an axis circularly or hyperbolically depending on
whether the rotation axis is timelike or spacelike, respectively.

\qquad In this paper, we investigate elliptical rotation matrices, which are
orthogonal matrices in the scalar product space, whose associated matrix is
diag$\left(  a_{1},a_{2},a_{3}\right)  $ with $a_{1},a_{2},a_{3}\in%
\mathbb{R}
^{+}$. First, we choose a proper scalar product to the given ellipse (or
ellipsoid) such that this ellipse (or ellipsoid) is equivalent to a circle (or
sphere) for the scalar product space. That is, the scalar product doesn't
change the distance between any point on the ellipse (or ellipsoid) and
origin. Interpreting a motion on an ellipsoid is an important concept since
planets usually have ellipsoidal shapes and elliptical orbits. The geometry of
ellipsoid can be examined using affine transformations, because of an
ellipsoid can be considered as an affine map of the unit sphere. For example,
for the ellipsoid\ $\mathcal{E}^{2}=\left\{  \mathbf{x}\in%
\mathbb{R}
^{3}:\mathbf{x}^{t}A\mathbf{x}\leq1\right\}  $ and the unit sphere
$S^{2}=\left\{  \mathbf{x}\in%
\mathbb{R}
^{3}:\left\Vert \mathbf{x}\right\Vert =\mathbf{x}^{t}\mathbf{x}\leq1\right\}
,$ we can write $\mathcal{E}^{2}=T\left(  S^{2}\right)  $ using the affine
transformation $T(\mathbf{x})=A\mathbf{x}+c$, $x\in\mathcal{E}^{2}.$ Then we
have, $\mathbf{Vol}(\mathcal{E}^{2})=\mathbf{Vol}(T(S^{2}))=\sqrt{\det
Q}\mathbf{Vol}\left(  S^{2}\right)  =\sqrt{\det Q}4\pi/3$ where $Q=AA^{t}.$

\qquad The aim of this study is to explain the motion on the ellipsoid
\[
\dfrac{x^{2}}{a^{2}}+\dfrac{y^{2}}{b^{2}}+\dfrac{z^{2}}{c^{2}}=1,
\]
as a rotation, using the proper inner product, vector product and elliptical
orthogonal matrices. In this method, the elliptical inner product, the vector
product and the angles are compatible with the parameters $\theta$ and $\beta$
of the parametrization $\varphi\left(  \theta,\beta\right)  =\left(
a\cos\theta\cos v,b\cos\theta\sin\beta,c\sin\theta\right)  $. We use the
classical methods to generate elliptical rotation matrices. In the
Preliminaries section, first we explain how to define a suitable scalar
product and a vector product for a given ellipsoid. Then we introduce the
symmetric, skew symmetric and orthogonal matrices in this elliptical scalar
product space. Finally, we examine the motion on an ellipsoid using elliptical
rotation matrices. In section 3, we generate the elliptical rotation matrices
using various classical methods (such as, Cayley formula, Rodrigues formula
and Householder transformation) compatible with the defined scalar product.
Furthermore, we defined the elliptic quaternions and generate elliptical
rotations using unit elliptic quaternions.

\section{Preliminaries}

\qquad We begin with a brief review of scalar products. More informations can
be found in (\cite{Macykey1}, \cite{Aragon1} and, \cite{affine}). Consider the map

\begin{center}
$\mathcal{B}:%
\mathbb{R}
^{n}\times%
\mathbb{R}
^{n}\rightarrow%
\mathbb{R}
,$ \ \ $\left(  \mathbf{u},\mathbf{v}\right)  \rightarrow\mathcal{B}\left(
\mathbf{u},\mathbf{v}\right)  $
\end{center}

for $\mathbf{u,v}\in%
\mathbb{R}
^{n}.$ If such a map is linear in each argument, that is,%
\begin{align*}
\mathcal{B}\left(  a\mathbf{u}+b\mathbf{v},\mathbf{w}\right)   &
=a\mathcal{B}\left(  \mathbf{u},\mathbf{w}\right)  +b\mathcal{B}\left(
\mathbf{v},\mathbf{w}\right)  ,\\
\mathcal{B}\left(  \mathbf{u},c\mathbf{v}+d\mathbf{w}\right)   &
=c\mathcal{B}\left(  \mathbf{u},\mathbf{v}\right)  +d\mathcal{B}\left(
\mathbf{u},\mathbf{w}\right)  ,
\end{align*}
where, $a,b,c,d\in%
\mathbb{R}
$ and $\mathbf{u,v,w}\in%
\mathbb{R}
^{n},$ then it is called a bilinear form. Given a bilinear form on $%
\mathbb{R}
^{n},$ there exists a unique $\Omega\in%
\mathbb{R}
^{n\times n}$ square matrix such that for all $\mathbf{u,v}\in%
\mathbb{R}
^{n},$ $\mathcal{B}\left(  \mathbf{u},\mathbf{v}\right)  =\mathbf{u}^{t}%
\Omega\mathbf{v}.$ $\Omega$ is called "\textit{the matrix associated with the
form}" with respect to the standard basis and we will denote $\mathcal{B}%
\left(  \mathbf{u},\mathbf{v}\right)  $ as $\mathcal{B}_{\Omega}\left(
\mathbf{u},\mathbf{v}\right)  $ as needed. A bilinear form is said to be
symmetric or skew symmetric if $\mathcal{B}\left(  \mathbf{u},\mathbf{v}%
\right)  =\mathcal{B}\left(  \mathbf{u},\mathbf{v}\right)  $ or $\mathcal{B}%
\left(  \mathbf{u},\mathbf{v}\right)  =-\mathcal{B}\left(  \mathbf{u}%
,\mathbf{v}\right)  $, respectively. Hence, the matrix associated with a
symmetric bilinear form is symmetric, and similarly, the associated matrix of
a skew symmetric bilinear form is skew symmetric. Also, a bilinear form is
nondegenerate if its associated matrix is non-singular. That is, for all
$\mathbf{u}\in%
\mathbb{R}
^{n}$, there exists $\mathbf{v}\in%
\mathbb{R}
^{n}$, such that $\mathcal{B}\left(  \mathbf{u},\mathbf{v}\right)  \neq0.$ A
real scalar product is a non-degenerate bilinear form. The space $%
\mathbb{R}
^{n}$ equipped with a fixed scalar product is said to be a real scalar product
space. Also, some scalar products, like the dot product, have positive
definitely property. That is, $\mathcal{B}\left(  \mathbf{u},\mathbf{u}%
\right)  \geq0$ and $\mathcal{B}\left(  \mathbf{u},\mathbf{u}\right)  =0$ if
and only if $\mathbf{u}=0.$ Now, we will define a positive definite scalar
product, which we call the $\mathcal{B}$-inner product or elliptical inner product.

\qquad Let $\mathbf{u}=\left(  u_{1},u_{2},...,u_{n}\right)  $, $\mathbf{w}%
=\left(  w_{1},w_{2},...,w_{n}\right)  \in%
\mathbb{R}
^{n}$ and $a_{1},a_{2},..,a_{n}\in%
\mathbb{R}
^{+}$. Then the map
\[
\mathcal{B}:%
\mathbb{R}
^{n}\times%
\mathbb{R}
^{n}\rightarrow%
\mathbb{R}
,\text{ \ \ }\mathcal{B}\left(  \mathbf{u},\mathbf{w}\right)  =a_{1}u_{1}%
w_{1}+a_{2}u_{2}w_{2}+\cdots+a_{n}u_{n}w_{n}%
\]
is a positive definite scalar product. We call it elliptical inner product or
$\mathcal{B}$-inner product. The real vector space $%
\mathbb{R}
^{n}$ equipped with the elliptical inner product will be represented by
$\mathbb{%
\mathbb{R}
}_{a_{1},a_{2},...,a_{n}}^{n}.$ Note that the scalar product $\mathcal{B}%
\left(  \mathbf{u},\mathbf{v}\right)  $ can be written as $\mathcal{B}\left(
\mathbf{u},\mathbf{w}\right)  =\mathbf{u}^{t}\Omega\mathbf{w}$ where
associated matrix is%
\begin{equation}
\Omega=\left[
\begin{array}
[c]{cccc}%
a_{1} & 0 & \cdots & 0\\
0 & a_{2} & \cdots & 0\\
\vdots & \vdots & \ddots & 0\\
0 & 0 & \cdots & a_{n}%
\end{array}
\right]  . \label{omega}%
\end{equation}
The number $\sqrt{\det\Omega}$ will be called "\textit{constant of the scalar
product}" and denoted by $\Delta$ in the rest of the paper. The norm of a
vector associated with the scalar product $\mathcal{B}$ is defined as
$\left\Vert \mathbf{u}\right\Vert _{\mathcal{B}}=\sqrt{\mathcal{B}\left(
\mathbf{u},\mathbf{u}\right)  }$. Two vectors $\mathbf{u}$ and $\mathbf{w}$
are called $\mathcal{B}$-orthogonal or elliptically orthogonal vectors if
$\mathcal{B}\left(  \mathbf{u},\mathbf{w}\right)  =0$. In addition, if their
norms are 1, then they are called $\mathcal{B}$-orthonormal vectors. If
$\left\{  \mathbf{u}_{1},\mathbf{u}_{2},...,\mathbf{u}_{n}\right\}  $ is an
$\mathcal{B}$-orthonormal base of $\mathbb{%
\mathbb{R}
}_{a_{1},a_{2},...,a_{n}}^{n}$, then $\det\left(  \mathbf{u}_{1}%
,\mathbf{u}_{2},...,\mathbf{u}_{n}\right)  =\Delta^{-1}.$ The cosine of the
angle between two vectors $\mathbf{u}$ and $\mathbf{w}$ is defined as,%
\[
\mathbf{cos}\theta=\dfrac{\mathcal{B}\left(  \mathbf{u},\mathbf{w}\right)
}{\left\Vert \mathbf{u}\right\Vert _{\mathcal{B}}\left\Vert \mathbf{w}%
\right\Vert _{\mathcal{B}}}%
\]
where $\theta$ is compatible with the parameters of the angular parametric
equations of ellipse or ellipsoid.

\bigskip

\qquad Let $\mathcal{B}$ be a non degenerate scalar product, $\Omega$ the
associated matrix of $\mathcal{B}$, and $R\in%
\mathbb{R}
^{n\times n}$ is any matrix.

\qquad\textbf{i)} If $\mathcal{B}\left(  R\mathbf{u},R\mathbf{w}\right)
=\mathcal{B}\left(  \mathbf{u},\mathbf{w}\right)  $ for all vectors
$\mathbf{u,w}\in\mathbb{%
\mathbb{R}
}^{n}$, then $R$ is called a $\mathcal{B}$-orthogonal matrix. It means that
orthogonal matrices preserve the norm of vectors and satisfy the matrix
equality $R^{t}\Omega R=\Omega$. Also, all rows (or columns) are $\mathcal{B}%
$-orthogonal to each other. We denote the set of $\mathcal{B}$-orthogonal
matrices by $\mathbf{O}_{\mathcal{B}}\left(  n\right)  $. That is,
\[
\mathbf{O}_{\mathcal{B}}\left(  n\right)  =\{R\in\mathbb{R}^{n\times n}%
:R^{t}\Omega R=\Omega\text{ and }\det R=\pm1\}.
\]
$\mathbf{O}_{\mathcal{B}}\left(  n\right)  $ is a subgroup of \textbf{Gl}%
$_{\mathcal{B}}\left(  n\right)  .$ It is sometimes called the isometry group
of $%
\mathbb{R}
^{n}$ associated with scalar product $\mathcal{B}$. The determinant of a
$\mathcal{B}$-orthogonal matrix can be either $-1$ or $1.$ If $\det R=1,$ then
we call it a $\mathcal{B}$-rotation matrix or an elliptical rotation matrix.
If $\det R=-1,$ we call it an elliptical reflection matrix. Although the set
$\mathbf{O}_{\mathcal{B}}\left(  n\right)  $ is not a linear subspace of $%
\mathbb{R}
^{n\times n},$ it is a Lie group. The isometry group for the bilinear or
sesquilinear forms can be found in \cite{Macykey1}. The set of the
$\mathcal{B}$-rotation matrices of $\mathbb{%
\mathbb{R}
}^{n}$ can be expressed as follows:%
\[
\mathbf{SO}_{\mathcal{B}}\left(  n\right)  =\{R\in\mathbb{R}^{n\times n}%
:R^{t}\Omega R=\Omega\text{ and }\det R=1\}.
\]
$\mathbf{SO}_{\mathcal{B}}\left(  n\right)  $ is a subgroup of $\mathbf{O}%
_{\mathcal{B}}\left(  n\right)  .$

\medskip

\qquad\textbf{ii)} If $\mathcal{B}\left(  S\mathbf{u},\mathbf{w}\right)
=\mathcal{B}\left(  \mathbf{u},S\mathbf{w}\right)  $ for all vectors
$\mathbf{u,w}\in\mathbb{%
\mathbb{R}
}^{n}$, then $S$ is called a $\mathcal{B}$-symmetric matrix. It satisfies
$S^{t}\Omega=\Omega S$. The set of $\mathcal{B}$-symmetric matrices, defined
by
\[
\mathbb{J}=\left\{  S\in%
\mathbb{R}
^{n\times n}:\mathcal{B}\left(  S\mathbf{u},\mathbf{w}\right)  =\mathcal{B}%
\left(  \mathbf{u},S\mathbf{w}\right)  \text{ for all }\mathbf{u,w}\in%
\mathbb{R}
^{n}\right\}
\]
is a Jordan algebra \cite{Macykey1}. It is a subspace of \ the vector space of
real $n\times n$ matrices, with dimension $n\left(  n+1\right)  /2$. Any
$\mathcal{B}$-symmetric matrix in $\mathbb{%
\mathbb{R}
}_{a_{1},a_{2},...,a_{n}}^{n}$ can be defined as
\begin{equation}
S=\left[  \dfrac{\Delta a_{ij}}{a_{i}}\right]  _{n\times n} \label{duz}%
\end{equation}
where $a_{ij}=a_{ji}$ and $a_{ij}\in%
\mathbb{R}
.$

\medskip

\qquad\textbf{iii)} If $\mathcal{B}\left(  T\mathbf{u},\mathbf{w}\right)
=-\mathcal{B}\left(  \mathbf{u},T\mathbf{w}\right)  $ for all vectors
$\mathbf{u,w}\in\mathbb{%
\mathbb{R}
}^{n}$, then $T$ is called a $\mathcal{B}$-skew-symmetric matrix. Also,
$T^{t}\Omega=-\Omega T$. The set of $\mathcal{B}$-skew symmetric matrices,
defined by
\[
\mathbb{L}=\left\{  T\in%
\mathbb{R}
^{n\times n}:\mathcal{B}\left(  T\mathbf{u},\mathbf{w}\right)  =-\mathcal{B}%
\left(  \mathbf{u},T\mathbf{w}\right)  \text{ for all }\mathbf{u,w}\in
V\right\}
\]
is a Lie algebra \cite{Macykey1}. It is a subspace of the vector space of real
$n\times n$ matrices, with dimension $n(n-1)/2$, as well. Any $\mathcal{B}%
$-skew-symmetric matrix in $\mathbb{%
\mathbb{R}
}_{a_{1},a_{2},...,a_{n}}^{n}$ can be defined as,%
\begin{equation}
T=\left[  t_{ij}\right]  _{n\times n}\text{ \ \ with \ \ }t_{ij}=\left\{
\begin{array}
[c]{cc}%
\dfrac{\Delta a_{ij}}{a_{i}}\smallskip & i>j\\
\dfrac{-\Delta a_{ij}}{a_{i}}\smallskip & i<j\\
0 & i=j
\end{array}
\right.  \label{ters}%
\end{equation}
where $a_{ij}=a_{ji}$ and $a_{ij}\in%
\mathbb{R}
.$

\bigskip

For example, in the scalar product space $%
\mathbb{R}
_{a_{1},a_{2},a_{3}}^{3},$ the symmetric and skew symmetric matrices are

\begin{center}
$S=\Delta\left[
\begin{array}
[c]{ccc}%
a_{11}/a_{1} & x/a_{1} & y/a_{1}\\
x/a_{2} & a_{22}/a_{2} & z/a_{2}\\
y/a_{3} & z/a_{3} & a_{33}/a_{3}%
\end{array}
\right]  $ \ \ and \ $T=\Delta\left[
\begin{array}
[c]{ccc}%
0 & x/a_{1} & y/a_{1}\\
-x/a_{2} & 0 & z/a_{2}\\
-y/a_{3} & -z/a_{3} & 0
\end{array}
\right]  .$
\end{center}

Note that, even if we omit the scalar product constant $\Delta$ in $S$ or $T,$
they will still be symmetric or skew symmetric matrix, respectively. But then,
we cannot generate elliptical rotation matrices using the Rodrigues and Cayley
formulas. So, we will keep the constant $\Delta$.

\bigskip

\qquad Now, we define the elliptical vector product, which is related to
elliptical inner product. Let \textbf{$u$}$_{i}=\left(  u_{i1},u_{i2}%
,...,u_{in}\right)  \in%
\mathbb{R}
^{n}$ for $i=1,2,...,n-1$ and $\mathbf{e}_{1},\mathbf{e}_{2},...,\mathbf{e}%
_{n}$ be standard unit vectors for $\mathcal{B}.$ Then, the\ elliptical vector
product in $\mathbb{%
\mathbb{R}
}_{a_{1},a_{2},...,a_{n}}^{n}$ is defined as,%
\begin{align*}
\mathbb{%
\mathbb{R}
}_{a_{1},a_{2},...,a_{n}}^{n}\times\mathbb{%
\mathbb{R}
}_{a_{1},a_{2},...,a_{n}}^{n}\times\cdots\times\mathbb{%
\mathbb{R}
}_{a_{1},a_{2},...,a_{n}}^{n}  &  \rightarrow\mathbb{%
\mathbb{R}
}_{a_{1},a_{2},...,a_{n}}^{n},\\
\left(  \mathbf{u}_{1},\mathbf{u}_{2},...,\mathbf{u}_{n}\right)   &
\rightarrow\mathcal{V}\left(  \mathbf{u}_{1}\times\mathbf{u}_{2}%
\times\mathbf{u}_{3}\times\cdots\times\mathbf{u}_{n-1}\right)
\end{align*}%
\begin{equation}
\mathcal{V}\left(  \mathbf{u}_{1}\times\mathbf{u}_{2}\times\mathbf{u}%
_{3}\times\cdots\times\mathbf{u}_{n-1}\right)  =\Delta\det\left[
\begin{array}
[c]{ccccc}%
\mathbf{e}_{1}/a_{1} & \mathbf{e}_{2}/a_{2} & \mathbf{e}_{3}/a_{3} & \cdots &
\mathbf{e}_{n}/a_{n}\\
u_{11} & u_{12} & u_{13} & \cdots & u_{1n}\\
u_{21} & u_{22} & u_{23} & \cdots & u_{2n}\\
\vdots & \vdots & \vdots & \vdots & \vdots\\
u_{\left(  n-1\right)  1} & u_{\left(  n-1\right)  2} & u_{\left(  n-1\right)
3} & \cdots & u_{\left(  n-1\right)  n}%
\end{array}
\right]  \label{vect}%
\end{equation}
The vector $\mathcal{V}\left(  \mathbf{u}_{1}\times\mathbf{u}_{2}%
\times\mathbf{u}_{3}\times\cdots\times\mathbf{u}_{n-1}\right)  $ is
$\mathcal{B}$-orthogonal to each of the vectors $\mathbf{u}_{1},\mathbf{u}%
_{2},\mathbf{u}_{3},...,\mathbf{u}_{n-1}.$

\section{Generating an Elliptical Rotation Matrix}

\qquad In this section, we will generate elliptical rotation matrices using
the elliptical versions of the classical methods. For a given ellipse in the
form
\begin{equation}
\left(  \mathbf{E}\right)  :\lambda a_{1}x^{2}+\lambda a_{2}y^{2}=1,\text{
\ \ \ }\lambda,a_{1},a_{2}\in%
\mathbb{R}
^{+} \label{ellipse}%
\end{equation}
we will use the scalar product $\mathcal{B}\left(  \mathbf{u},\mathbf{w}%
\right)  =a_{1}u_{1}w_{1}+a_{2}u_{2}w_{2}.$ That is, our scalar product space
is $%
\mathbb{R}
_{a_{1},a_{2}}^{2}.$ An elliptical rotation matrix represents a rotation on
$\left(  \mathbf{E}\right)  $ or any ellipse similar to $\left(
\mathbf{E}\right)  .$ Recall that ellipses with the same eccentricity are
called similar. Since the shape of an ellipse depends only on the ratio
$a_{1}/a_{2},$ $\lambda$ in Equation (\ref{ellipse}), does not affect the
rotation matrix.

\subsection{Rodrigues Rotation Formula}

$\qquad$The Rodrigues rotation formula is a useful method for generating
rotation matrices. Given a rotation axis and an angle, we can readily generate
a rotation matrix using this method. $\mathbf{SO}(n)$ is a Lie group and the
space of skew-symmetric matrices of dimension $n$ is the Lie algebra of
$\mathbf{SO}(n)$. We denote this Lie algebra by $\mathfrak{so}$$\left(
n\right)  .$ The exponential map defined by the standard matrix exponential
series $e^{A}$ connects it to the Lie group. For any skew-symmetric matrix A,
the matrix exponential $e^{A}$ always gives a rotation matrix. This method is
known as the Rodrigues formula.

\subsubsection{Elliptical Rotations In the Plane}

\qquad According to $\left(  \text{\ref{ters}}\right)  $ a skew symmetric
matrix can be expressed as
\[
T=\left[
\begin{array}
[c]{cc}%
0 & -\sqrt{a_{2}}/\sqrt{a_{1}}\\
\sqrt{a_{1}}/\sqrt{a_{2}} & 0
\end{array}
\right]  .
\]

The equality $T^{t}\Omega=-\Omega T$ is satisfied. The characteristic
polynomial of $T$ is, $P\left(  x\right)  =x^{2}+1.$ So, $T^{2}=-I$. We can
obtain the elliptical rotation matrix easily using the matrix exponential.

\begin{theorem}
Let $T$ be a $\mathcal{B}$-skew symmetric matrix. Then, the matrix
exponential
\[
R_{\theta}^{\mathcal{B}}=e^{T\theta}=I+\left(  \sin\theta\right)  T+\left(
1-\cos\theta\right)  T^{2}=\left[
\begin{array}
[c]{cc}%
\cos\theta & -\dfrac{\sqrt{a_{2}}}{\sqrt{a_{1}}}\sin\theta\\
\dfrac{\sqrt{a_{1}}}{\sqrt{a_{2}}}\sin\theta & \cos\theta
\end{array}
\right]
\]
gives an elliptical rotation along the ellipse $\lambda a_{1}x^{2}+\lambda
a_{2}y^{2}=1,$ $\lambda,a_{1},a_{2}\in%
\mathbb{R}
^{+}.$ That is, $R_{\theta}^{\mathcal{B}}$ is an elliptical rotation matrix in
the space $%
\mathbb{R}
_{a_{1},a_{2}}^{3}.\allowbreak$
\end{theorem}

\setcounter{theorem}{0}

\begin{remark}
All similar ellipses have identical elliptical rotation matrices.
\end{remark}

\setcounter{theorem}{0}

\begin{example}
Let's consider the ellipse $\left(  \mathbf{E}_{1}\right)  :\dfrac{x^{2}}%
{9}+\dfrac{y^{2}}{4}=1$ with the parametrization $\alpha\left(  \theta\right)
=\left(  3\cos\theta,2\sin\theta\right)  $. Let's take the points
\[
A=\alpha\left(  \pi/4\right)  =\left(  3\sqrt{2}/2,\sqrt{2}\right)  \text{
\ and \ }B=\alpha\left(  \pi/4+\pi/3\right)  =(3\sqrt{2}/4-3\sqrt{6}%
/4,\sqrt{2}/2+\sqrt{6}/2).
\]
That is, if we rotate the point $A$ elliptically through the angle $\pi/3,$ we
get $B.$ Now, let's find elliptical rotation matrix for these ellipses using
the above theorem and get same results. To calculate the elliptical rotation
matrix, first, we choose the elliptical inner product
\[
\mathcal{B}\left(  \mathbf{u},\mathbf{w}\right)  =\dfrac{u_{1}w_{1}}{9}%
+\dfrac{u_{2}w_{2}}{4}.
\]
in accordance with the ellipse $\left(  \mathbf{E}_{1}\right)  $ such that
$\mathbf{u}=\left(  u_{1},u_{2}\right)  $ and $\mathbf{w}=\left(  w_{1}%
,w_{2}\right)  $. \ That is, our space is $%
\mathbb{R}
_{1/9,1/4}^{2}$, $\Delta=1/6$ and the $\mathcal{B}$-skew symmetric matrix has
the form%
\[
T=\left[
\begin{array}
[c]{cc}%
0 & -3/2\\
2/3 & 0
\end{array}
\right]  .
\]
Note that, $T^{2}=-I.$ We can obtain elliptical rotation matrix as,
\[
R_{\theta}^{\mathcal{B}}=e^{\theta T}=\left[
\begin{array}
[c]{cc}%
\cos\theta & -\dfrac{3\sin\theta}{2}\\
\dfrac{2\sin\theta}{3} & \cos\theta
\end{array}
\right]  .
\]
where $R_{\theta}^{\mathcal{B}}$ is a $\mathcal{B}$-orthogonal matrix in $%
\mathbb{R}
_{1/9,1/4}^{2}.$ Namely, the equalities $\det R_{\theta}^{\mathcal{B}}=1$ and
$\left(  R_{\theta}^{\mathcal{B}}\right)  ^{t}\Omega\left(  R_{\theta
}^{\mathcal{B}}\right)  =\Omega$ are satisfied. \ For $\theta=\pi/3,$ we get
\[
R_{\pi/3}^{\mathcal{B}}=\left[
\begin{array}
[c]{cc}%
1/2 & -3\sqrt{3}/4\\
\sqrt{3}/3 & 1/2
\end{array}
\right]  .
\]
So, if we rotate the point $A$ elliptically, we get $B=R_{\theta}%
^{\mathcal{B}}\left(  A\right)  =(3\sqrt{2}/4-3\sqrt{6}/4,\sqrt{2}/2+\sqrt
{6}/2).$ Thus, we get same result using elliptical rotation matrix for
$\left(  \mathbf{E}_{1}\right)  .$ Note that, $\left\Vert R_{\theta
}^{\mathcal{B}}\left(  A\right)  \right\Vert _{\mathcal{B}}=1$ and the angle
between $\mathbf{x=}\overrightarrow{OA}$ and $\mathbf{y=}\overrightarrow{OB}$
is
\[
\cos\theta=\dfrac{\mathcal{B}\left(  \mathbf{x,y}\right)  }{\left\Vert
\mathbf{x}\right\Vert _{\mathcal{B}}\left\Vert \mathbf{y}\right\Vert
_{\mathcal{B}}}=\dfrac{1}{2}.
\]
It can be seen that the elliptical rotation matrix $R_{\theta}^{\mathcal{B}}$
can be also used to interpret the motion on a similar ellipse $\dfrac{x^{2}%
}{36}+\dfrac{y^{2}}{16}=1$ to $\left(  \mathbf{E}_{1}\right)  .$
\end{example}

\begin{figure}[ptb]
\centering
\includegraphics{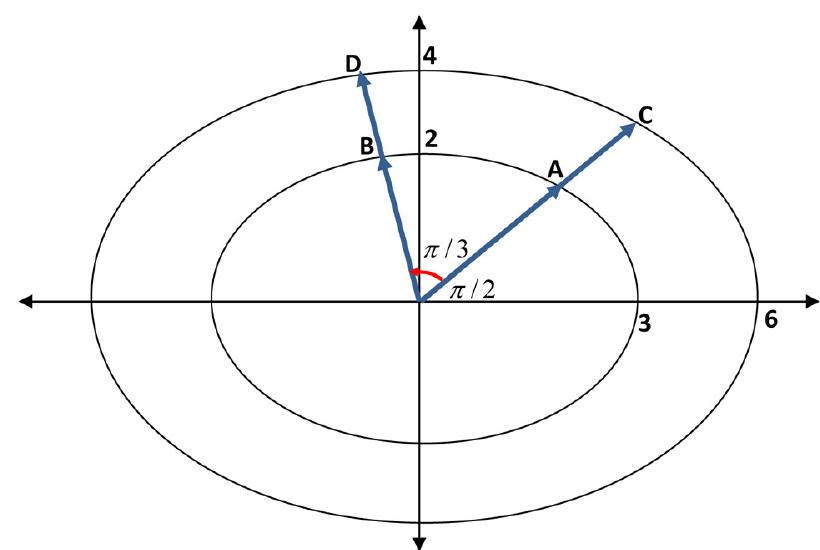} \caption{$\alpha\left(  \theta\right)  =\left(
3\cos\theta,2\sin\theta\right)  $ and $\beta\left(  \theta\right)  =\left(
6\cos\theta,4\sin\theta\right)  $}%
\label{fig:elephant}%
\end{figure}

\subsubsection{3-Dimensional Elliptical Rotations}

\qquad Let's take the ellipsoid $a_{1}x^{2}+a_{2}y^{2}+a_{3}z^{2}=1.$ The
scalar product for this ellipsoid is
\[
\mathcal{B}\left(  \mathbf{u},\mathbf{w}\right)  =a_{1}u_{1}w_{1}+a_{2}%
u_{2}w_{2}+a_{3}u_{1}w_{3},
\]
for $\mathbf{u}=\left(  u_{1},u_{2},u_{3}\right)  $ and $\mathbf{v=}\left(
v_{1},v_{2},v_{3}\right)  .$ Also, the vector product is

\begin{center}
$\mathcal{V}\left(  \mathbf{u}\times\mathbf{v}\right)  $=$\Delta\det\left[
\begin{array}
[c]{ccc}%
\mathbf{e}_{1}/a_{1} & \mathbf{e}_{2}/a_{2} & \mathbf{e}_{3}/a_{3}\\
u_{1} & u_{2} & u_{3}\\
v_{1} & v_{2} & v_{3}%
\end{array}
\right]  $=$\Delta\left[
\begin{array}
[c]{ccc}%
0 & -u_{3}/a_{1} & u_{2}/a_{1}\\
u_{3}/a_{2} & 0 & -u_{1}/a_{2}\\
-u_{2}/a_{3} & u_{1}/a_{3} & 0
\end{array}
\right]  \left[
\begin{array}
[c]{c}%
v_{1}\\
v_{2}\\
v_{3}%
\end{array}
\right]  $=$T\left(  \mathbf{v}^{t}\right)  $
\end{center}

where $\Delta=\sqrt{a_{1}a_{2}a_{3}}.$ The matrix%
\begin{equation}
T=\Delta\left[
\begin{array}
[c]{ccc}%
0 & -u_{3}/a_{1} & u_{2}/a_{1}\\
u_{3}/a_{2} & 0 & -u_{1}/a_{2}\\
-u_{2}/a_{3} & u_{1}/a_{3} & 0
\end{array}
\right]  \label{3ters}%
\end{equation}
is skew symmetric in $%
\mathbb{R}
_{a_{1},a_{2},a_{3}}^{3}.$ That is, $T^{t}\Omega=-\Omega T$. So, the vector
product in $%
\mathbb{R}
_{a_{1},a_{2},a_{3}}^{3}$ can be viewed as a linear transformation, which
corresponds to multiplication by\ a skew symmetric matrix. The characteristic
polynomial of $T$ is, $P\left(  x\right)  =x^{3}+\left\Vert \mathbf{u}%
\right\Vert ^{2}x$ \ whose eigenvalues are $x_{1}=0$ and $x_{2,3}%
=\pm\left\Vert \mathbf{u}\right\Vert i.$ According to characteristic
polynomial $T^{3}+\left\Vert \mathbf{u}\right\Vert ^{2}T=0.$ So, if we take a
unit vector $\mathbf{u}\in%
\mathbb{R}
_{a_{1},a_{2},a_{3}}^{3},$ we get $T^{3}=-T$ and we can use Rodrigues and
Cayley formulas.

\bigskip

\begin{theorem}
Let $T$ be a skew symmetric matrix in the form $\left(  \text{\ref{3ters}%
}\right)  $ such that $\mathbf{u}=\left(  u_{1},u_{2},u_{3}\right)  \in%
\mathbb{R}
_{a_{1},a_{2},a_{3}}^{3}$ is a unit vector. Then, the matrix exponential
\[
R_{\theta}^{\mathcal{B},\mathbf{u}}=e^{T\theta}=I+\left(  \sin\theta\right)
T+\left(  1-\cos\theta\right)  T^{2}%
\]
gives an elliptical rotation on the ellipsoid $a_{1}x^{2}+a_{2}y^{2}%
+a_{3}z^{2}=1.$ Furthermore, the matrix $R_{\theta}^{\mathcal{B},\mathbf{u}}$
can be expressed as
\begin{equation}
\left[
\begin{array}
[c]{ccc}%
a_{1}u_{1}^{2}+\left(  1-a_{1}u_{1}^{2}\right)  \cos\theta & -\tfrac{\Delta
u_{3}\sin\theta}{a_{1}}-a_{2}u_{1}u_{2}\left(  \cos\theta-1\right)  &
\tfrac{\Delta u_{2}\sin\theta}{a_{1}}-a_{3}u_{1}u_{3}\left(  \cos
\theta-1\right) \\
\tfrac{\Delta u_{3}\sin\theta}{a_{2}}-a_{1}u_{1}u_{2}\left(  \cos
\theta-1\right)  & a_{2}u_{2}^{2}+\left(  1-a_{2}u_{2}^{2}\right)  \cos\theta
& -\tfrac{\Delta u_{1}\sin\theta}{a_{2}}-a_{3}u_{2}u_{3}\left(  \cos
\theta-1\right) \\
-\tfrac{\Delta u_{2}\sin\theta}{a_{3}}-a_{1}u_{1}u_{3}\left(  \cos
\theta-1\right)  & \tfrac{\Delta u_{1}\sin\theta}{a_{3}}-a_{2}u_{2}%
u_{3}\left(  \cos\theta-1\right)  & a_{3}u_{3}^{2}+\left(  1-a_{3}u_{3}%
^{2}\right)  \cos\theta
\end{array}
\right]  \label{rodmat}%
\end{equation}
where $\mathbf{u}=\left(  u_{1},u_{2},u_{3}\right)  $ is the rotation axis and
$\theta$ is the elliptical rotation angle.
\end{theorem}

\begin{proof}
Since $\mathbf{u}$ is a unit vector in $%
\mathbb{R}
_{a_{1},a_{2},a_{3}}^{3}$, we have $T^{3}=-T$. So, we get
\begin{align*}
R_{\theta}^{\mathcal{B},\mathbf{u}}  &  =e^{\theta T}=I+\theta T+\dfrac
{\theta^{2}T^{2}}{2!}+\dfrac{-\theta^{3}T}{3!}+\dfrac{-\theta^{4}T^{2}}%
{4!}+\dfrac{\theta^{5}T}{5!}+\dfrac{\theta^{6}T^{2}}{6!}\cdots\\
&  =I+T\left(  \theta-\dfrac{\theta^{3}}{3!}+\dfrac{\theta^{5}}{5!}%
-\cdots\right)  \text{+}T^{2}\left(  \dfrac{\theta^{2}}{2!}-\dfrac{\theta^{4}%
}{4!}+\dfrac{\theta^{6}}{6!}-\cdots\right) \\
&  =I+T\left(  \theta-\dfrac{\theta^{3}}{3!}+\dfrac{\theta^{5}}{5!}%
-\cdots\right)  \text{+}T^{2}\left(  1-\left(  1-\dfrac{\theta^{2}}{2!}%
+\dfrac{\theta^{4}}{4!}-\dfrac{\theta^{6}}{6!}+\cdots\right)  \right) \\
&  =I+\left(  \sin\theta\right)  T+\left(  1-\cos\theta\right)  T^{2}.
\end{align*}
If we expand this formula using $-a_{2}u_{2}^{2}-a_{3}u_{3}^{2}=a_{1}u_{1}%
^{2}-1,$ we can obtain the rotation matrix as (\ref{rodmat}).
\end{proof}

\begin{figure}[ptbh]
\centering \includegraphics[scale=0.9]{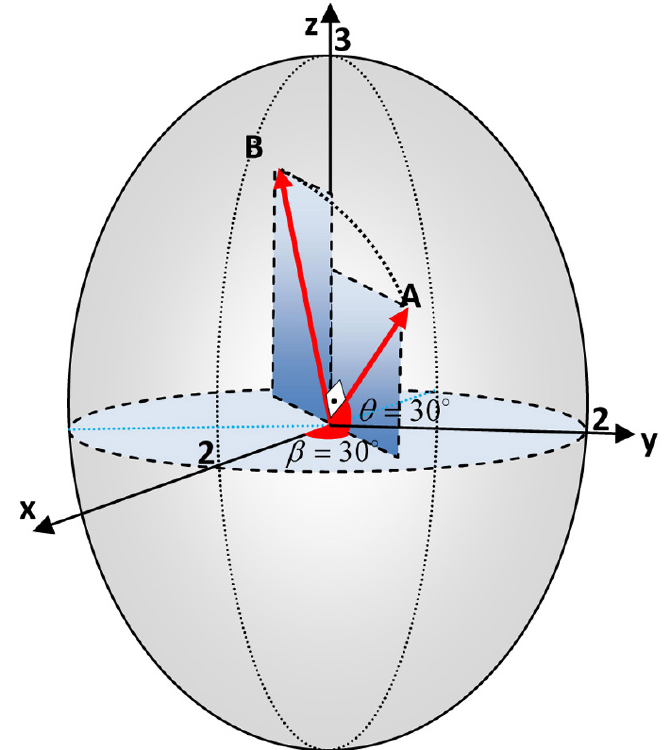}\caption{{}}%
\end{figure}

\setcounter{theorem}{1}

\begin{example}
A parametrization of the ellipsoid $\dfrac{x^{2}}{4}+\dfrac{y^{2}}{4}%
+\dfrac{z^{2}}{9}=1$ is
\[
\alpha\left(  \theta,\beta\right)  =\left(  2\cos\theta\cos\beta,2\cos
\theta\sin\beta,3\sin\theta\right)
\]
where $\theta\in\lbrack0,\pi)$ and $\beta\in\lbrack0,2\pi).$ Let's take the
points
\[
A\text{=}\alpha\left(  30^{\circ},30^{\circ}\right)  \text{=}\left(
3/2,\sqrt{3}/2,3/2\right)  \text{ \ and \ }B\text{=}\alpha\left(  120^{\circ
},30^{\circ}\right)  \text{=}\left(  -\sqrt{3}/2,-1/2,3\sqrt{3}/2\right)
\]
on the ellipsoid. Let's find the rotation matrix which is rotate the point $A$
to $B$ elliptically. We have $a_{1}=a_{2}=1/4$ and $a_{3}=1/9.$ So,
$\Delta=1/12.$ First, using the vector product of $\mathbf{x}=\overrightarrow
{OA}$ and $\mathbf{y}=\overrightarrow{OB}$ in $\mathbb{%
\mathbb{R}
}_{1/4,1/4,1/9}^{3},$ we find the rotation axis $\mathbf{u}.$
\[
\mathcal{V}\left(  \mathbf{x}\times\mathbf{y}\right)  =\dfrac{1}%
{12}\left\vert
\begin{array}
[c]{ccc}%
4i & 4j & 9k\\
3/2 & \sqrt{3}/2 & 3/2\\
-\sqrt{3}/2 & -\dfrac{1}{2} & 3\sqrt{3}/2
\end{array}
\right\vert =\left(  1,-\sqrt{3},0\right)  .
\]
Since $\mathcal{V}\left(  \mathbf{x}\times\mathbf{y}\right)  $ is unit vector
in $\mathbb{%
\mathbb{R}
}_{1/4,1/4,1/9}^{3},$ we get $\mathbf{u=(}1,-\sqrt{3},0).$ Thus, we obtain the
elliptical rotation matrix%
\[
R_{\theta}^{\mathbf{u}}T\left(  \theta\right)  =\dfrac{1}{12}\left[
\begin{array}
[c]{ccc}%
9\cos\theta+3 & 3\sqrt{3}\left(  \cos\theta-1\right)  & -4\sqrt{3}\sin\theta\\
3\sqrt{3}\left(  \cos\theta-1\right)  & 3\cos\theta+9 & -4\sin\theta\\
9\sqrt{3}\sin\theta & 9\sin\theta & 12\cos\theta
\end{array}
\allowbreak\right]
\]
by using (\ref{rodmat}). This matrix describes an elliptical rotation on a
great ellipse such that it is intersection of the ellipsoid and the plane
passing through the origin and $\mathbf{\mathcal{B}}$-orthogonal to
$\mathbf{u.}$ It can be easily found that equation of the plane is $x=\sqrt
{3}y.$ So, $R_{\theta}^{\mathbf{u}}$ represents an elliptical rotation over
the the great ellipse is $y^{2}+\dfrac{1}{9}z^{2}=1,$ $y=\sqrt{3}x$. Also, the
elliptical rotation angle is $\pi/2,$ since $\cos\theta=\mathcal{B}\left(
\mathbf{x},\mathbf{y}\right)  =0$ (Figure 2b). Thus, we find
\begin{equation}
R_{\pi/2}^{\mathbf{u}}=\dfrac{1}{12}\left[
\begin{array}
[c]{ccc}%
3 & -3\sqrt{3} & -4\sqrt{3}\\
-3\sqrt{3} & 9 & -4\\
9\sqrt{3} & 9 & 0
\end{array}
\right]  . \label{examp2}%
\end{equation}
The matrix (\ref{examp2}), rotates the point $A$ to the point $B$ elliptically
over the great ellipse $\dfrac{x^{2}}{2}+\dfrac{z^{2}}{9}=1,$ $y=x.$
\end{example}

\setcounter{theorem}{1}

\begin{remark}
The eigenvalues of the matrix (\ref{rodmat}) are, $x_{1}=e^{i\theta},$
$x_{2}=e^{-i\theta}$ and $x_{3}=1$. Also, the eigenvector corresponding to $1$
is the vector $\mathbf{u}$, the rotation axis of the motion.
\end{remark}

\subsection{Elliptical Cayley Rotation Matrix}

\qquad In 1846, Arthur Cayley discovered a formula to express special
orthogonal matrices using skew-symmetric matrices. It is called Cayley
rotation matrix. In this section, we will describe the Cayley rotation matrix
for any ellipsoid whose scalar product is $\mathcal{B}$. We call it
$\mathcal{B}$-Cayley rotation matrix.

\qquad The Cayley rotation matrix is a useful tool that gives a
parameterization for rotation matrices without the need to use trigonometric
functions. Let $A$ be an $n\times n$ skew symmetric matrix where $\left(
I-A\right)  $ is invertible. Cayley formula transforms the matrix $A$ into
$\left(  I+A\right)  \left(  I-A\right)  ^{-1}.$ \textbf{SO}$(n)$ is
isomorphic to $\mathfrak{so}$$(n)$ via the Cayley formula where $\mathfrak{so}%
$$(n)$ denotes the space of skew symmetric matrices, usually associated with
the Lie algebra of the transformation group defined by rotations in
\textbf{SO}$(n)$. That is, the Cayley map is defined as%
\begin{align*}
C  &  :\text{$\mathfrak{so}$}\left(  n\right)  \rightarrow\mathbf{SO}\left(
n\right) \\
A  &  \rightarrow C\left(  A\right)  =\left(  I+A\right)  \left(  I-A\right)
^{-1}.
\end{align*}
$\ \ \ \ \ \ \ \ \ \ \ \ \ \ \ \ \ \ \ \ \ \ \ $

\qquad Note that, if a matrix $M$ has the eigenvalue $-1,$ then $\left(
I-M\right)  $ is not invertible. But, since $A$ is skew symmetric, all its
eigenvalues are purely imaginary in the Euclidean space and $\left(
I-A\right)  $ is invertible. That is, the Cayley formula is well-defined for
all skew symmetric matrices. The inverse of the Cayley map is given by
\[
C^{-1}=\left(  I-A\right)  \left(  I+A\right)  ^{-1}.
\]

\qquad In the Minkowski space, an eigenvalue of a skew symmetric matrix can be
$-1.$ In this case, the Cayley formula is not valid. For detailed information
on the Cayley formula in Minkowski space see \cite{melek2}.

\begin{theorem}
Let $T$ be a skew symmetric matrix in the form (\ref{3ters}) such that
$\mathbf{u}=\left(  u_{1},u_{2},u_{3}\right)  \in%
\mathbb{R}
_{a_{1},a_{2},a_{3}}^{3}$. Then,
\[
R^{\mathcal{B},\mathbf{u}}=\left(  I+T\right)  \left(  I-T\right)  ^{-1}%
\]
is an elliptical rotation matrix on the ellipsoid $a_{1}x^{2}+a_{2}y^{2}%
+a_{3}z^{2}=1$ where $\mathbf{u}$ is the rotation axis. Furthermore, the
matrix $R^{\mathcal{B},\mathbf{u}}$ can be written in the form
\begin{equation}
\tfrac{1}{1+\left\Vert \mathbf{u}\right\Vert _{\mathcal{B}}^{2}}\left[
\begin{array}
[c]{ccc}%
a_{1}u_{1}^{2}{\small -}u_{2}^{2}a_{2}{\small -}u_{3}^{2}a_{3}{\small +}1 &
a_{2}u_{1}u_{2}{\small -}\tfrac{2\Delta u_{3}}{a_{1}}{\small +}u_{1}u_{2}a_{2}
& \tfrac{2\Delta u_{2}}{a_{1}}{\small +}a_{3}u_{1}u_{3}{\small +}u_{1}%
u_{3}a_{3}\\
\tfrac{2\Delta u_{3}}{a_{2}}{\small +}a_{1}u_{1}u_{2}{\small +}u_{1}u_{2}a_{1}
& a_{2}u_{2}^{2}{\small -}u_{1}^{2}a_{1}{\small -}u_{3}^{2}a_{3}{\small +}1 &
a_{3}u_{2}u_{3}{\small -}\tfrac{2\Delta u_{1}}{a_{2}}{\small +}u_{2}u_{3}%
a_{3}\\
a_{1}u_{1}u_{3}{\small -}\tfrac{2\Delta u_{2}}{a_{3}}{\small +}u_{1}u_{3}a_{1}
& \tfrac{2\Delta u_{1}}{a_{3}}{\small +}a_{2}u_{2}u_{3}{\small +}u_{2}%
u_{3}a_{2} & a_{3}u_{3}^{2}{\small -}u_{1}^{2}a_{1}{\small -}u_{2}^{2}%
a_{2}{\small +}1
\end{array}
\right]  . \label{cayleyrot}%
\end{equation}

\end{theorem}

\begin{proof}
Since, $T$ is a $\mathcal{B}$-skew symmetric matrix, we have $T^{t}%
\Omega=-\Omega T.$ Also, we can write
\[
\left(  I+T\right)  ^{t}\Omega=\Omega\left(  I-T\right)  \text{ and }\left(
I-T\right)  ^{t}\Omega=\Omega\left(  I+T\right)  .
\]

Using these equalities, it can be seen that
\[
\left(  R^{\mathcal{B},\mathbf{u}}\right)  ^{t}\Omega\left(  R^{\mathcal{B}%
,\mathbf{u}}\right)  =\left(  \left(  I+T\right)  \left(  I-T\right)
^{-1}\right)  ^{t}\Omega\left(  I+T\right)  \left(  I-T\right)  ^{-1}=\Omega.
\]
Also, since $\det\left(  I+T\right)  =1+\left\Vert \mathbf{u}\right\Vert
_{\mathcal{B}}$ and $\det\left(  I-T\right)  ^{-1}=\dfrac{1}{1+\left\Vert
\mathbf{u}\right\Vert _{\mathcal{B}}},$ we have $\det R^{\mathcal{B}%
,\mathbf{u}}=1.$ That is, $R^{\mathcal{B},\mathbf{u}}$ is an elliptical
rotation matrix. The matrix (\ref{cayleyrot}) can be obtained after some
tedious computations.
\end{proof}

\setcounter{theorem}{2}

\begin{remark}
The eigenvalues of the matrix (\ref{cayleyrot}) are
\[
\lambda_{1}=\dfrac{1-\left\Vert \mathbf{u}\right\Vert _{\mathcal{B}}^{2}%
}{1+\left\Vert \mathbf{u}\right\Vert _{\mathcal{B}}^{2}}+\dfrac{2\left\Vert
\mathbf{u}\right\Vert _{\mathcal{B}}i}{1+\left\Vert \mathbf{u}\right\Vert
_{\mathcal{B}}^{2}},\text{ \ \ \ }\lambda_{2}=\dfrac{1-\left\Vert
\mathbf{u}\right\Vert _{\mathcal{B}}^{2}}{1+\left\Vert \mathbf{u}\right\Vert
_{\mathcal{B}}^{2}}-\dfrac{2\left\Vert \mathbf{u}\right\Vert _{\mathcal{B}}%
i}{1+\left\Vert \mathbf{u}\right\Vert _{\mathcal{B}}^{2}}\text{ \ \ and
\ \ }\lambda_{3}=1.
\]
Also, the eigenvector corresponding to $1$ is $\mathbf{u,}$ which is\ the
rotation axis. The matrix (\ref{cayleyrot}) rotates a vector elliptically on
the ellipsoid $a_{1}x^{2}+a_{2}y^{2}+a_{3}z^{2}=1$ about the axis
$\mathbf{u}=\left(  u_{1},u_{2},u_{3}\right)  $ through the elliptical angle
$\theta$ where
\begin{equation}
\tan\theta=\dfrac{2\left\Vert \mathbf{u}\right\Vert _{\mathcal{B}}%
}{1-\left\Vert \mathbf{u}\right\Vert _{\mathcal{B}}^{2}}. \label{tantetat}%
\end{equation}

\end{remark}

\setcounter{theorem}{2}

\begin{example}
Let's find the elliptical rotation matrix representing a elliptical rotation
on the ellipsoid $\dfrac{x^{2}}{4}+\dfrac{y^{2}}{9}+z^{2}=1$ about the axis
$\mathbf{u}=\left(  2,3,1\right)  .$ Using the matrix (\ref{cayleyrot}), we
find
\[
\left[
\begin{array}
[c]{ccc}%
0 & 0 & 2\\
3/2 & 0 & 0\\
0 & 1/3 & 0
\end{array}
\right]  .
\]
The elliptical rotation angle corresponding to this matrix, can be found\ to
be $-\pi/3$ from the formula (\ref{tantetat}).
\end{example}

\subsection{Elliptical Householder Transformation}

\qquad The Householder transformation was introduced in 1958 by Alston Scott
Householder. A Householder transformation\ is a linear transformation in the form

\begin{center}
$\mathcal{H}_{\mathbf{v}}\left(  x\right)  =x-\dfrac{2\mathbf{vv}^{t}%
}{\mathbf{v}^{t}\mathbf{v}}x$
\end{center}

where $\mathbf{v}$ is a nonzero vector. This transformation describes a
reflection about a plane or hyperplane passing through the origin and
orthogonal to $\mathbf{v.}$ Householder transformations on spaces with a
non-degenerate bilinear or sesquilinear forms are studied in \cite{Macykey1}.
Every orthogonal transformation is the combination of reflections with respect
to hyperplanes. This is known as the Cartan--Dieudonn\'{e} theorem. A
constructive proof of the Cartan--Dieudonn\'{e} theorem for the case of
generalized scalar product spaces is given by Uhlig \cite{Uhlig} and Fuller
\cite{Fuller}. An alternative proof of the Cartan--Dieudonn\'{e} theorem for
generalized real scalar product spaces\ is given by Rodr\'{\i}guez-Andrade and
etc. \cite{Aragon1}. They used the Clifford algebras to compute the
factorization of a given orthogonal transformation as a product of reflections.

\qquad Householder transformations are widely used in tridiagonalization of
symmetric matrices and to perform QR decompositions in numerical linear
algebra \cite{J. Gallier1}.\ Householder Transformation is cited in the top 10
algorithms of the 20th century \cite{yuzyil}. Generalized Householder matrices
are the simplest generalized orthogonal matrices. A generalized Householder,
or $\mathcal{B}$-Householder, matrix has the form

\begin{center}
$\mathcal{H}_{\mathbf{v}}\left(  x\right)  =x-\dfrac{2\mathbf{vv}^{t}\Omega
}{\mathbf{v}^{t}\Omega\mathbf{v}}x$
\end{center}

for a non-B-isotropic vector $\mathbf{v}$ (i.e., $\mathbf{v}^{t}%
\Omega\mathbf{v}\neq0$) \cite{Fuller}, \cite{Aragon1}.

\qquad In this paper, we use the elliptical version of the Householder
transformation to generate an elliptical rotation matrix in a real elliptical
scalar product space $\mathbb{%
\mathbb{R}
}_{a_{1},a_{2},...,a_{n}}^{n}.$ Let $\mathcal{H}_{\mathbf{v}}=\left[
h_{ij}\right]  _{n\times n}$ be a Householder matrix. Then

\begin{center}
$h_{ij}=\delta_{ij}-\dfrac{2v_{i}v_{j}\Omega_{jj}}{\mathbf{v}^{t}%
\Omega\mathbf{v}}$
\end{center}

where $\delta_{ij}$ is the Kronecker delta.\textbf{ }The Householder matrix
$\mathcal{H}_{\mathbf{v}}$\ is a $\mathcal{B}$-symmetric and $\mathcal{B}%
$-orthogonal matrix. To prove these, first we will show that $\mathcal{H}%
_{\mathbf{v}}^{t}\Omega=\Omega\mathcal{H}_{\mathbf{v}}$ : $\mathbf{v}%
^{t}\Omega\mathbf{v}=\mathcal{B}\left(  \mathbf{v},\mathbf{v}\right)
=\left\Vert \mathbf{v}\right\Vert _{\mathcal{B}}^{2}$ is a positive real
number$.$ So,

\begin{center}
$\mathcal{H}_{\mathbf{v}}^{t}\Omega=\left(  I-\dfrac{2\mathbf{vv}^{t}\Omega
}{\mathbf{v}^{t}\Omega\mathbf{v}}\right)  ^{t}\Omega=\left(  \Omega
-\dfrac{2\Omega\mathbf{vv}^{t}\Omega}{\mathbf{v}^{t}\Omega\mathbf{v}}\right)
=\Omega\left(  I-\dfrac{2\mathbf{vv}^{t}\Omega}{\mathbf{v}^{t}\Omega
\mathbf{v}}\right)  =\Omega\mathcal{H}_{\mathbf{v}},$
\end{center}

and $\mathcal{H}_{\mathbf{v}}$ is a $\mathbf{\mathcal{B}}$-symmetric. Next, we
will show that $\mathcal{H}_{\mathbf{v}}^{t}\Omega\mathcal{H}_{\mathbf{v}%
}=\Omega.$
\begin{align*}
\mathcal{H}_{\mathbf{v}}^{t}\Omega\mathcal{H}_{\mathbf{v}}  &  =\Omega
\mathcal{H}_{\mathbf{v}}^{2}\text{ \ (}\mathcal{H}_{\mathbf{v}}\text{ is
$B$-symmetric)}\\
&  =\Omega\left(  I-\dfrac{2\mathbf{vv}^{t}\Omega}{\mathbf{v}^{t}%
\Omega\mathbf{v}}\right)  \left(  I-\dfrac{2\mathbf{vv}^{t}\Omega}%
{\mathbf{v}^{t}\Omega\mathbf{v}}\right) \\
&  =\Omega\left(  I-\dfrac{2\mathbf{vv}^{t}\Omega}{\mathbf{v}^{t}%
\Omega\mathbf{v}}-\dfrac{2\mathbf{vv}^{t}\Omega}{\mathbf{v}^{t}\Omega
\mathbf{v}}+\dfrac{4\left(  \mathbf{vv}^{t}\Omega\right)  \left(
\mathbf{vv}^{t}\Omega\right)  }{\left(  \mathbf{v}^{t}\Omega\mathbf{v}\right)
^{2}}\right) \\
&  =\Omega.
\end{align*}
So, $\mathcal{H}_{\mathbf{v}}$\ is $\mathcal{B}$-orthogonal.

\bigskip

\qquad Using the above transformations, we can write the elliptical reflection
matrix as%
\begin{equation}
\mathcal{H}_{\mathbf{v}}=\dfrac{1}{\left\Vert \mathbf{v}\right\Vert
_{\mathcal{B}}^{2}}\left[
\begin{array}
[c]{ccc}%
\left\Vert \mathbf{v}\right\Vert _{\mathcal{B}}^{2}-2a_{1}v_{1}^{2} &
-2a_{2}v_{1}v_{2} & -2a_{3}v_{1}v_{3}\\
-2a_{1}v_{2}v_{1} & \left\Vert \mathbf{v}\right\Vert _{\mathcal{B}}^{2}%
-2a_{2}v_{2}^{2} & -2a_{3}v_{2}v_{3}\\
-2a_{1}v_{3}v_{1} & -2a_{2}v_{3}v_{2} & \left\Vert \mathbf{v}\right\Vert
_{\mathcal{B}}^{2}-2a_{3}v_{3}^{2}%
\end{array}
\right]  \label{house1}%
\end{equation}
for the ellipsoid $a_{1}x^{2}+a_{2}y^{2}+a_{3}z^{2}=1$.

\setcounter{theorem}{0}

\begin{corollary}
The Householder transformation doesn't change the length of a vector in $%
\mathbb{R}
_{a_{1},a_{2},a_{3}}^{n}$. It describes an elliptical reflection about a plane
passing through the origin $\mathbf{\mathcal{B}}$-orthogonal to $\mathbf{v.}$
Using the matrix (\ref{house1}), we can express elliptical reflections on any ellipsoid.
\end{corollary}

\setcounter{theorem}{3}

\begin{example}
Let's find the elliptical reflection matrix which reflects a point
elliptically on the ellipsoid $2x^{2}+2y^{2}+z^{2}=1,$ about the plane which
is passes through the origin and is $\mathcal{B}$-orthogonal to $\mathbf{v}%
=\left(  1,2,3\right)  .$ Our elliptical inner product is $\mathcal{B}\left(
x,y\right)  =2x_{1}y_{2}+2x_{2}y_{2}+x_{3}y_{3}$. So, $\left\Vert
\mathbf{v}\right\Vert _{\mathcal{B}}^{2}=2\cdot1^{2}+2\cdot2^{2}+1\cdot
3^{2}=\allowbreak19$ and we get
\[
\mathcal{H}_{v}=\left[
\begin{array}
[c]{ccc}%
15/19 & -8/19 & -6/19\\
-8/19 & 3/19 & -12/19\\
-12/19 & -24/19 & 1/19
\end{array}
\right]  .
\]
The determinant of this matrix is $-1$ and the equality $\mathcal{H}%
_{\mathbf{v}}^{t}\Omega\mathcal{H}_{\mathbf{v}}=\Omega$ holds. We find the
reflection of $A\left(  1/2,1/2,0\right)  $ to be $B(7/38,-5/38,-18/19).$
Let's check it : The middle point of $\left[  AB\right]  $ is $C\left(
13/38,7/38,-9/19\right)  .$ The equation of the plane is $\mathcal{B}\left(
\mathbf{v},\mathbf{x}\right)  =0$ where $\mathbf{x}=\left(  x,y,z\right)  .$
In our case, it is $2x+4y+3z=0,$ so $C$ lies on this plane. Also, it can be
easily seen that $\overrightarrow{AB}$ is $\mathcal{B}$-orthogonal to the plane.
\end{example}

\setcounter{theorem}{3}

\begin{remark}
For the $n$ dimensional ellipsoid associated with the scalar product
$\mathcal{B}$, the elliptical reflection matrix $\mathcal{H}_{\mathbf{v}}$ is
in the form
\[
\mathcal{H}_{\mathbf{v}}^{\mathcal{B}}=\dfrac{1}{\left\Vert \mathbf{v}%
\right\Vert _{\mathcal{B}}^{2}}\left[
\begin{array}
[c]{ccccc}%
\left\Vert \mathbf{v}\right\Vert _{\mathcal{B}}^{2}-2a_{1}v_{1}^{2} &
-2a_{2}v_{1}v_{2} & -2a_{3}v_{1}v_{3} & \cdots & -2a_{n}v_{1}v_{n}\\
-2a_{1}v_{2}v_{1} & \left\Vert \mathbf{v}\right\Vert _{\mathcal{B}}^{2}%
-2a_{2}v_{2}^{2} & -2a_{3}v_{2}v_{3} & \cdots & -2a_{n}v_{2}v_{n}\\
-2a_{1}v_{3}v_{1} & -2a_{2}v_{3}v_{2} & \left\Vert \mathbf{v}\right\Vert
_{\mathcal{B}}^{2}-2a_{3}v_{3}^{2} & \cdots & -2a_{n}v_{3}v_{n}\\
\vdots & \vdots & \vdots & \ddots & \vdots\\
-2a_{1}v_{n}v_{1} & -2a_{2}v_{n}v_{2} & -2v_{n}a_{3}v_{3} & \cdots &
\left\Vert \mathbf{v}\right\Vert _{\mathcal{B}}^{2}-2a_{n}v_{n}^{2}%
\end{array}
\right]
\]
where $\mathbf{v}^{t}=\left(  v_{1},v_{2},...,v_{n}\right)  .$
\end{remark}

\begin{theorem}
Let $\boldsymbol{x,y}\in%
\mathbb{R}
_{\mathcal{B}}^{n}$ be any two nonzero vectors where $\left\Vert
\boldsymbol{x}\right\Vert _{\mathcal{B}}=\left\Vert \boldsymbol{y}\right\Vert
_{\mathcal{B}}$. If $\mathbf{v}=\boldsymbol{x}-\boldsymbol{y}$, then%
\[
\mathcal{H}_{\mathbf{v}}^{\mathcal{B}}\left(  \mathbf{x}\right)
=\boldsymbol{x}-\dfrac{2\mathbf{vv}^{t}\Omega}{\mathbf{v}^{t}\Omega\mathbf{v}%
}\boldsymbol{x}=\boldsymbol{y}.
\]

\end{theorem}

\begin{proof}
By direct computation, we find%
\[
\mathbf{v}^{t}\Omega\mathbf{v}=\mathcal{B}\left(  \boldsymbol{x}%
-\boldsymbol{y,x}-\boldsymbol{y}\right)  =2\left\Vert \boldsymbol{x}%
\right\Vert _{\mathcal{B}}^{2}-2\boldsymbol{y}^{t}\Omega\boldsymbol{x}%
\]
and%
\[
2\mathbf{vv}^{t}\Omega\boldsymbol{x}=2\left(  \boldsymbol{x}-\boldsymbol{y}%
\right)  \left(  \boldsymbol{x}-\boldsymbol{y}\right)  ^{t}\Omega
\boldsymbol{x}=\left(  \boldsymbol{x}-\boldsymbol{y}\right)  \left(
2\left\Vert \boldsymbol{x}\right\Vert _{\mathcal{B}}^{2}-2\boldsymbol{y}%
^{t}\Omega\boldsymbol{x}\right)  .
\]
Using these equalities, we obtain%
\[
\mathcal{H}_{\mathbf{v}}^{\mathcal{B}}\left(  \mathbf{x}\right)
=\boldsymbol{x}-\dfrac{2\mathbf{vv}^{t}\Omega}{\mathbf{v}^{t}\Omega\mathbf{v}%
}\boldsymbol{x=x-}\dfrac{\left(  \boldsymbol{x}-\boldsymbol{y}\right)  \left(
2\left\Vert \boldsymbol{x}\right\Vert _{\mathcal{B}}^{2}-2\boldsymbol{y}%
^{t}\Omega\boldsymbol{x}\right)  }{2\left\Vert \boldsymbol{x}\right\Vert
_{\mathcal{B}}^{2}-2\boldsymbol{y}^{t}\Omega\boldsymbol{x}}=\boldsymbol{x-x+y}%
=\boldsymbol{y.}%
\]

\end{proof}

\bigskip

\begin{theorem}
Let $\boldsymbol{x,y}\in%
\mathbb{R}
_{\mathcal{B}}^{n}$ be any two nonzero vectors where $\left\Vert
\boldsymbol{x}\right\Vert _{\mathcal{B}}=\left\Vert \boldsymbol{y}\right\Vert
_{\mathcal{B}}$. If $\mathbf{v}=\boldsymbol{x}+\boldsymbol{y}$, then%
\[
\mathcal{H}_{\mathbf{v}}^{\mathcal{B}}\left(  \boldsymbol{x}\right)
=-\boldsymbol{y}.
\]

\end{theorem}

\begin{proof}
By direct computation, we see that%
\[
\mathbf{v}^{t}\Omega\mathbf{v}=\mathcal{B}\left(  \boldsymbol{x}%
+\boldsymbol{y,x}+\boldsymbol{y}\right)  =2\left\Vert \boldsymbol{x}%
\right\Vert _{\mathcal{B}}^{2}+2\boldsymbol{y}^{t}\Omega\boldsymbol{x}%
\]
and%
\[
2\mathbf{vv}^{t}\Omega\boldsymbol{x}=2\left(  \boldsymbol{x}+\boldsymbol{y}%
\right)  \left(  \boldsymbol{x}+\boldsymbol{y}\right)  ^{t}\Omega
\boldsymbol{x}=\left(  \boldsymbol{x}+\boldsymbol{y}\right)  \left(
2\left\Vert \boldsymbol{x}\right\Vert _{\mathcal{B}}^{2}+2\boldsymbol{y}%
^{t}\Omega\boldsymbol{x}\right)  .
\]
Thus, we get%
\[
\mathcal{H}_{\mathbf{v}}^{\mathcal{B}}\left(  \boldsymbol{x}\right)
\boldsymbol{=x-}\left(  \boldsymbol{x+y}\right)  \boldsymbol{=-y}.
\]
Moreover, we have $\mathcal{H}_{\boldsymbol{y}}^{\mathcal{B}}\left(
-\boldsymbol{y}\right)  =\boldsymbol{y}.$%
\[
\mathcal{H}_{\mathbf{y}}^{\mathcal{B}}\left(  \boldsymbol{-y}\right)
=-\boldsymbol{y}+\dfrac{2\boldsymbol{yy}^{t}\Omega}{\boldsymbol{y}^{t}%
\Omega\boldsymbol{y}}\boldsymbol{y=-y+\dfrac{2\boldsymbol{y}\left(
\boldsymbol{y}^{t}\Omega\boldsymbol{y}\right)  }{\boldsymbol{y}^{t}%
\Omega\boldsymbol{y}}=y.}%
\]
which means that $\mathcal{H}_{\boldsymbol{y}}^{\mathcal{B}}\mathcal{H}%
_{\boldsymbol{x+y}}^{\mathcal{B}}\left(  \boldsymbol{x}\right)
=\boldsymbol{y}.$ Hence, we obtain the elliptical rotation matrix%
\[
R_{\mathbf{\mathcal{B}}}=\mathcal{H}_{\boldsymbol{y}}^{\mathcal{B}}%
\mathcal{H}_{\boldsymbol{x+y}}^{\mathcal{B}}%
\]
which rotates $\mathbf{x}$ to $\mathbf{y}$ elliptically. The rotation axis is
the vector $\mathcal{V}\left(  \mathbf{x}\times\mathbf{y}\right)  $ and\ the
elliptic rotation angle can be found using elliptical inner product.
\end{proof}

\setcounter{theorem}{5}

\begin{example}
Let's find the elliptic rotation matrix that rotates $\boldsymbol{x}=\left(
0,0,5\right)  $ to $\boldsymbol{y}=\left(  2,2,3\right)  $ on the ellipsoid
$2x^{2}+2y^{2}+z^{2}=1.$ Using the matrix (\ref{house1}), we get
\[
\mathcal{H}_{\boldsymbol{x+y}}=\dfrac{1}{5}\left[
\begin{array}
[c]{ccc}%
4 & -1 & -2\\
-1 & 4 & -2\\
-4 & -4 & -3
\end{array}
\right]  \text{ \ and \ }\mathcal{H}_{\boldsymbol{y}}=\dfrac{1}{25}\left[
\begin{array}
[c]{ccc}%
9 & -16 & -12\\
-16 & 9 & -12\\
-24 & -24 & 7
\end{array}
\right]  .
\]
Thus, the rotation matrix is
\[
R=\mathcal{H}_{\boldsymbol{y}}\mathcal{H}_{\boldsymbol{x+y}}=\dfrac{1}%
{5}\left[
\begin{array}
[c]{ccc}%
4 & -1 & 2\\
-1 & 4 & 2\\
-4 & -4 & 3
\end{array}
\right]
\]
which satisfies $R\left(  \boldsymbol{x}\right)  =\boldsymbol{y}$, $\det R=1,$
and $R^{t}\Omega R=\Omega$ where $\Omega=diag\left(  2,2,1\right)  .$ Note
that, we can obtain same matrix using (\ref{rodmat}).
\end{example}

\subsection{Elliptic Quaternions}

\qquad Quaternions were discovered by Sir William R. Hamilton in 1843 and the
theory of quaternions was expanded to include applications such as rotations
in the early 20th century. The most important property of the quaternions is
that every unit quaternion represents a rotation and this plays an important
role in the study of rotations in 3-dimensional vector spaces. Using unit
quaternions is a useful, natural, and elegant way to perceive rotations.
Quaternions are used especially in computer vision, computer graphics,
animation, and kinematics.

\qquad Quaternion algebra $\mathbb{H}$ is an associative, non-commutative
division ring with four basic elements $\{1,\mathbf{i},\mathbf{j}%
,\mathbf{k}\}$ satisfying the equalities $\mathbf{i}^{2}=\mathbf{j}%
^{2}=\mathbf{k}^{2}=\mathbf{ijk}=-1.$ We can express any quaternion $q$ as
$q=\left(  q_{1},q_{2},q_{3},q_{4}\right)  =q_{1}+q_{2}\mathbf{i}%
+q_{3}\mathbf{j}+q_{4}\mathbf{k}$ or $q=S_{q}+\mathbf{V}q$ where the symbols
$S_{q}=q_{1}$ and $\mathbf{V}_{q}=q_{2}\mathbf{i}+q_{3}\mathbf{j}%
+q_{4}\mathbf{k}$ denote the scalar and vector parts of $q,$ respectively$.$
If $S_{q}=0$ then $q$ is called a pure quaternion. The conjugate of $q$ is
denoted by $\overline{q},$ and defined as $\overline{q}=S_{q}-\mathbf{V}_{q}$.
The norm of a quaternion $q=q_{0}+q_{1}\mathbf{i}+q_{2}\mathbf{j}%
+q_{3}\mathbf{k}$ is defined by $\sqrt{q\overline{q}}=\sqrt{\overline{q}%
q}=\sqrt{q_{0}^{2}+q_{1}^{2}+q_{2}^{2}+q_{3}^{2}}$ and is denoted by $N_{q}$.
We say that $q_{0}=q/N_{q}$ is unit quaternion if $q\neq0.$ The set of unit
quaternions are denoted by $\mathbb{H}_{1}.$ Every unit quaternion can be
written in the form $q_{0}=\cos\theta+\boldsymbol{\varepsilon}_{0}\sin\theta$
where $\boldsymbol{\varepsilon}_{0}$ is a unit vector satisfying the equality
$\boldsymbol{\varepsilon}_{0}^{2}=-1.$It is called the axis of the quaternion
\cite{qua2}, \cite{qua3}.

\qquad Let $q$ and $r$ be two quaternions. Then, the linear transformation
$R_{q}:\mathbb{H}\rightarrow\mathbb{H}$ defined by $R_{q}\left(  r\right)
=qrq^{-1}$ is a quaternion that has the same norm and scalar as $r$. Since the
scalar part of the quaternion $r$ doesn't change under $R_{q},$ we will only
examine how its vector part $\mathbf{V}_{r}$ changes under the transformation
$R_{q}.$ We can interpret the rotation of a vector in the Euclidean 3-space
using the quaternion product $q\mathbf{V}_{r}q^{-1}$.

\qquad If $q=q_{0}+q_{1}\mathbf{i}+q_{2}\mathbf{j}+q_{3}\mathbf{k}=\cos
\theta+\boldsymbol{\varepsilon}_{0}\sin\theta$ is a unit quaternion, then,
using the linear transformation $R_{q}\left(  \mathbf{V}_{r}\right)
=q\mathbf{V}_{r}q^{-1}$, the corresponding rotation matrix can be found as
\begin{equation}
R_{q}=\left[
\begin{array}
[c]{ccc}%
q_{0}^{2}+q_{1}^{2}-q_{2}^{2}-q_{3}^{2} & -2q_{0}q_{3}+2q_{1}q_{2} &
2q_{0}q_{2}+2q_{1}q_{3}\\
2q_{1}q_{2}+2q_{3}q_{0} & q_{0}^{2}-q_{1}^{2}+q_{2}^{2}-q_{3}^{2} &
2q_{2}q_{3}-2q_{1}q_{0}\\
2q_{1}q_{3}-2q_{2}q_{0} & 2q_{1}q_{0}+2q_{2}q_{3} & q_{0}^{2}-q_{1}^{2}%
-q_{2}^{2}+q_{3}^{2}%
\end{array}
\right]  . \tag{4}%
\end{equation}
This rotation matrix represents a rotation through angle $2\theta$ about the
axis $\boldsymbol{\varepsilon}=\left(  q_{1},q_{2},q_{3}\right)  .$ In the
Lorentzian space, the rotation matrix corresponding to a timelike quaternion
$q=q_{0}+q_{1}\mathbf{i}+q_{2}\mathbf{j}+q_{3}\mathbf{k}$ is,
\[
R_{q}=\left[
\begin{array}
[c]{ccc}%
q_{1}^{2}+q_{2}^{2}+q_{3}^{2}+q_{4}^{2} & 2q_{1}q_{4}-2q_{2}q_{3} &
-2q_{1}q_{3}-2q_{2}q_{4}\\
2q_{2}q_{3}+2q_{4}q_{1} & q_{1}^{2}-q_{2}^{2}-q_{3}^{2}+q_{4}^{2} &
-2q_{3}q_{4}-2q_{2}q_{1}\\
2q_{2}q_{4}-2q_{3}q_{1} & 2q_{2}q_{1}-2q_{3}q_{4} & q_{1}^{2}-q_{2}^{2}%
+q_{3}^{2}-q_{4}^{2}%
\end{array}
\right]
\]
where the set of unit timelike quaternions satisfies the properties
$q\overline{q}=1,$ $\mathbf{i}^{2}=-1,$ $\mathbf{j}^{2}=\mathbf{k}%
^{2}=\mathbf{ijk}=1$ (see $\cite{qua4}).$

\subsubsection{Elliptic Quaternions}

\qquad To get an elliptical rotation matrix, firs we define the set of
elliptic quaternions suitable for the ellipsoid $a_{1}x^{2}+a_{2}y^{2}%
+a_{3}z^{2}=1$. Let's take four basic elements $\{1,\mathbf{i},\mathbf{j}%
,\mathbf{k}\}$ satisfying the equalities

\begin{center}
$\mathbf{i}^{2}=-a_{1},$ $\ \ \ \ \mathbf{j}^{2}=-a_{2},$ $\ \ \ \ \mathbf{k}%
^{2}=-a_{3}$
\end{center}

and

\begin{center}
$\mathbf{ij}=\dfrac{\Delta}{a_{3}}\mathbf{k=-ji},$ $\ \ \ \ \ \mathbf{jk}%
=\dfrac{\Delta}{a_{1}}\mathbf{i=-kj},$ $\ \ \ \ \ \ \ \mathbf{ki}%
=\dfrac{\Delta}{a_{2}}\mathbf{j=-ik}$
\end{center}

where $a_{1},a_{2},a_{3}\in%
\mathbb{R}
^{+}$ and $\Delta=\sqrt{a_{1}a_{2}a_{3}}\mathbf{.}$

\bigskip

\qquad The set of elliptic quaternions will be denoted by $\mathbb{H}%
_{a_{1},a_{2},a_{3}}.$ This set is an associative, non-commutative division
ring with our basic elements $\{1,\mathbf{i},\mathbf{j},\mathbf{k}\}.$ If we
take $a_{1}=a_{2}=a_{3}=1,$ we get the usual quaternion algebra. The elliptic
quaternion product table is given below.

\begin{center}%
\begin{tabular}
[c]{|c||c|c|c|c|}\hline
& $\mathbf{1}$ & $\mathbf{i}$ & $\mathbf{j}$ & $\mathbf{k}$\\\hline\hline
$\mathbf{1}$ & $\mathbf{1}$ & $\mathbf{i}$ & $\mathbf{j}$ & $\mathbf{k}%
$\\\hline
$\mathbf{i}$ & $\mathbf{i}$ & $-a_{1}$ & $\Delta\mathbf{k/}a_{3}$ &
$-\Delta\mathbf{j/}a_{2}$\\\hline
$\mathbf{j}$ & $\mathbf{j}$ & $\mathbf{-}\Delta\mathbf{k/}a_{3}$ & $-a_{2}$ &
$\Delta\mathbf{i/}a_{1}$\\\hline
$\mathbf{k}$ & $\mathbf{k}$ & $\Delta\mathbf{j/}a_{2}$ & $-\Delta
\mathbf{i/}a_{1}$ & $-a_{3}$\\\hline
\end{tabular}

\end{center}

\bigskip

\qquad The elliptic quaternion product of two quaternions $p=p_{0}%
+p_{1}\mathbf{i}+p_{2}\mathbf{j}+p_{3}\mathbf{k}$ and $q=q_{0}+q_{1}%
\mathbf{i}+q_{2}\mathbf{j}+q_{3}\mathbf{k}$ is defined as%
\begin{equation}
p_{0}q_{0}-\mathcal{B}\left(  \mathbf{V}_{p},\mathbf{V}_{q}\right)
+p_{0}\mathbf{V}_{q}+q_{0}\mathbf{V}_{p}+\mathcal{V}\left(  \mathbf{V}%
_{p}\times\mathbf{V}_{q}\right)  \label{vectin}%
\end{equation}
where $\mathcal{B}\left(  \mathbf{V}_{p},\mathbf{V}_{q}\right)  $ and
$\mathcal{V}\left(  \mathbf{V}_{p}\times\mathbf{V}_{q}\right)  $ are the
elliptical scalar product and\ the elliptical vector product, respectively. If
$p$ and $q$ are pure, then%
\begin{align*}
pq  &  =-\mathcal{B}\left(  \mathbf{V}_{p},\mathbf{V}_{q}\right)
+\mathcal{V}\left(  \mathbf{V}_{p}\times\mathbf{V}_{q}\right) \\
&  =-\left(  a_{1}p_{1}q_{1}+a_{2}p_{2}q_{2}+a_{3}p_{3}q_{3}\right)
+\Delta\left\vert
\begin{array}
[c]{ccc}%
i/a_{1} & j/a_{2} & k/a_{3}\\
p_{1} & p_{2} & p_{3}\\
q_{1} & q_{2} & q_{3}%
\end{array}
\right\vert .
\end{align*}
The elliptic quaternion product for $\mathbb{H}_{a_{1},a_{2},a_{3}}$ can be
expressed as

\begin{center}
$pq=\left[
\begin{array}
[c]{cccc}%
p_{0} & -a_{1}p_{1} & -a_{2}p_{2} & -a_{3}p_{3}\\
p_{1} & p_{0} & -\dfrac{p_{3}\Delta}{a_{1}} & \dfrac{p_{2}\Delta}{a_{1}}\\
p_{2} & \dfrac{p_{3}\Delta}{a_{2}} & p_{0} & -\dfrac{p_{1}\Delta}{a_{3}}\\
p_{3} & -\dfrac{p_{2}\Delta}{a_{3}} & \dfrac{p_{1}\Delta}{a_{3}} & p_{0}%
\end{array}
\right]  \left[
\begin{array}
[c]{c}%
q_{0}\\
q_{1}\\
q_{2}\\
q_{3}%
\end{array}
\right]  .\allowbreak$
\end{center}

For example, let $p,q\in\mathbb{H}_{2,2,1}$. Then, the elliptic quaternion
product of $p$ and $q$ defined is

\begin{center}
$pq=\left[
\begin{array}
[c]{cccc}%
p_{0} & -2p_{1} & -2p_{2} & -p_{3}\\
p_{1} & p_{0} & -p_{3} & p_{2}\\
p_{2} & p_{3} & p_{0} & -p_{1}\\
p_{3} & -2p_{2} & 2p_{1} & p_{0}%
\end{array}
\right]  \allowbreak\allowbreak\left[
\begin{array}
[c]{c}%
q_{0}\\
q_{1}\\
q_{2}\\
q_{3}%
\end{array}
\right]  .$
\end{center}

For $p=1+2\mathbf{i}+3\mathbf{j}+4\mathbf{k}$ and $q=2+4\mathbf{i}%
+\mathbf{j}+3\mathbf{k},$ we get $pq=\left(  -32,13,17,-9\right)  .$ This can
also be calculated using the product table

\begin{center}%
\begin{tabular}
[c]{|c||c|c|c|c|}\hline
& $\mathbf{1}$ & $\mathbf{i}$ & $\mathbf{j}$ & $\mathbf{k}$\\\hline\hline
$\mathbf{1}$ & $\mathbf{1}$ & $\mathbf{i}$ & $\mathbf{j}$ & $\mathbf{k}%
$\\\hline
$\mathbf{i}$ & $\mathbf{i}$ & $-2$ & $2\mathbf{k}$ & $-\mathbf{j}$\\\hline
$\mathbf{j}$ & $\mathbf{j}$ & $\mathbf{-}2\mathbf{k}$ & $-2$ & $\mathbf{i}%
$\\\hline
$\mathbf{k}$ & $\mathbf{k}$ & $\mathbf{j}$ & $-\mathbf{i}$ & $-1$\\\hline
\end{tabular}
.
\end{center}

Remember that the algebra is formed by a vector space $\mathbb{V}$ equipped
with a quadratic form $Q$ with the following equalities%
\begin{align*}
\mathbf{v}^{2}  &  =Q\left(  \mathbf{v}\right)  ;\\
\mathbf{uv+vu}  &  =2\mathcal{B}_{Q}\left(  \mathbf{u},\mathbf{v}\right)
\end{align*}
is called a \textit{Clifford algebra }and is denoted by $C\ell\left(
\mathbb{V},Q\right)  .$ If $\left\{  \mathbf{e}_{1},...,\mathbf{e}%
_{n}\right\}  $ is a base for an n-dimensional vector space $\mathbb{V},$ then
$\emph{C}\mathbf{\ell}\left(  \mathbb{V},Q\right)  $ is formed by the
multivectors
\[
\left\{  1\right\}  \cup\left\{  \mathbf{e}_{i_{1}}\mathbf{e}_{i_{2}%
}...\mathbf{e}_{i_{k}}:1\leq i_{1}\leq...\leq i_{k}\leq n,\text{ }1\leq k\leq
n\right\}
\]
with $\dim\left(  C\ell\left(  \mathbb{V},Q\right)  \right)  =2^{n}.$ Since
the Clifford product of two even multivectors is an even multivector, they
define an even subalgebra of $C\ell\left(  \mathbb{V},Q\right)  $. The even
subalgebra of an $n$-dimensional Clifford algebra is isomorphic to a Clifford
algebra of $(n-1)$ dimensions and it is denoted by $C\ell^{+}\left(
\mathbb{V},Q\right)  .$ The Hamiltonian quaternion algebra $\mathbb{H}$ is
isomorphic with the even subalgebra $C\ell_{3,0}^{+}=C\ell\left(
\mathbb{R}
^{3},Q=x_{1}^{2}+x_{2}^{2}+x_{3}^{2}\right)  $ by $\{1,e_{2}e_{3}%
\rightarrow\mathbf{j},e_{1}e_{3}\rightarrow\mathbf{k},e_{1}e_{2}%
\rightarrow\mathbf{i}\}$ and the split quaternion algebra $\widehat
{\mathbb{H}}$ is isomorphic with the even subalgebra $C\ell_{2,1}^{+}%
=C\ell\left(
\mathbb{R}
_{1}^{3},Q=-x_{1}^{2}+x_{2}^{2}+x_{3}^{2}\right)  $ by $\{1,e_{2}%
e_{3}\rightarrow\mathbf{i},e_{3}e_{1}\rightarrow\mathbf{k},e_{1}%
e_{2}\rightarrow\mathbf{j}\}$ \cite{qua2}$.$ Similarly, the elliptic
quaternion algebra is an even subalgebra of the Clifford algebra

\begin{center}
$\emph{C}\mathbf{\ell}\left(
\mathbb{R}
^{3}\right)  =\left\{  q=q_{0}+\mathbf{e}_{1}q_{1}+\mathbf{e}_{2}%
q_{2}+\mathbf{e}_{3}q_{3}:\mathbf{e}_{1}^{2}=a_{1},\text{ }\mathbf{e}_{2}%
^{2}=a_{2},\text{ }\mathbf{e}_{3}^{2}=a_{3},\text{ }\mathbf{e}_{i}%
\mathbf{e}_{j}+\mathbf{e}_{j}\mathbf{e}_{i}=0\right\}  $
\end{center}

associated with the nondegenerate quadratic form $Q\left(  x\right)
=a_{1}x_{1}^{2}+a_{2}x_{2}^{2}+a_{3}x_{3}^{2}$ and is denoted by
$\emph{C}\mathbf{\ell}^{+}\left(
\mathbb{R}
^{3},a_{1}x_{1}^{2}+a_{2}x_{2}^{2}+a_{3}x_{3}^{2}\right)  ,$ or shortly
$\emph{C}\mathbf{\ell}^{+}\left(
\mathbb{R}
_{a_{1},a_{2},a_{3}}^{3}\right)  .$ $\mathbb{H}_{a_{1},a_{2},a_{3}}$ is
isomorphic to $\emph{C}\mathbf{\ell}^{+}\left(
\mathbb{R}
_{a_{1},a_{2},a_{3}}^{3}\right)  $ with $\{1,\dfrac{a_{1}}{\Delta}e_{2}%
e_{3}\rightarrow\mathbf{i},\dfrac{a_{2}}{\Delta}e_{1}e_{3}\rightarrow
\mathbf{j},\dfrac{a_{3}}{\Delta}e_{1}e_{2}\rightarrow\mathbf{k}\}.$

For the quadratic form $Q=a_{1}x_{1}^{2}+a_{1}x_{2}^{2}+a_{1}x_{3}^{2},$
recall that the elliptical inner product can be obtained by using the
equality
\[
\mathcal{B}_{Q}\left(  \mathbf{x},\mathbf{y}\right)  =\frac{1}{2}\left[
Q\left(  \mathbf{x+y}\right)  -Q\left(  \mathbf{x}\right)  -Q\left(
\mathbf{y}\right)  \right]  .
\]
So, we get $\mathcal{B}_{Q}\left(  \mathbf{x},\mathbf{y}\right)  =a_{1}%
x_{1}y_{1}+a_{2}x_{2}y_{2}+a_{3}x_{3}y_{3}$ for $\mathbf{x=}\left(
x_{1},x_{2},x_{3}\right)  $ and $\mathbf{y=}\left(  y_{1},y_{2},y_{3}\right)
.$ Thus, we can construct an elliptic quaternion algebra for any elliptical
inner product space.

Conjugate, norm and inverse of an elliptic quaternion $q=q_{0}+q_{1}%
\mathbf{i}+q_{2}\mathbf{j}+q_{3}\mathbf{k}$ can be defined similar to usual
quaternions :
\[
\overline{q}=q_{0}-q_{1}\mathbf{i}-q_{2}\mathbf{j}-q_{3}\mathbf{k,}%
\ \ \ N_{q}=\sqrt{q\overline{q}}=\sqrt{\overline{q}q}=\sqrt{q_{0}^{2}%
+a_{1}q_{1}^{2}+a_{2}q_{2}^{2}+a_{3}q_{3}^{2}},\text{ \ \ }q^{-1}%
=\dfrac{\overline{q}}{N_{q}}.
\]
Also, each elliptic quaternion $q=q_{0}+q_{1}\mathbf{i}+q_{2}\mathbf{j}%
+q_{3}\mathbf{k}$ can be written in the form
\[
q_{0}=N_{q}\left(  \cos\theta+\boldsymbol{\varepsilon}_{0}\sin\theta\right)
\]
where

\begin{center}
$\cos\theta=\dfrac{q_{0}}{N_{q}}$ \ \ and \ $\sin\theta=\dfrac{\sqrt
{a_{1}q_{1}^{2}+a_{2}q_{2}^{2}+a_{3}q_{3}^{2}}}{N_{q}}.$
\end{center}

Here, $\boldsymbol{\varepsilon}_{0}=\dfrac{\left(  q_{1},q_{2},q_{3}\right)
}{\sqrt{a_{1}q_{1}^{2}+a_{2}q_{2}^{2}+a_{3}q_{3}^{2}}}$ is a unit vector in
the scalar product space $%
\mathbb{R}
_{a_{1},a_{2},a_{3}}^{3}$ satisfying the equality $\boldsymbol{\varepsilon
}_{0}^{2}=-1$. It is called the axis of the rotation. For example, if
$q=1+2\mathbf{i}+\mathbf{j}+5\mathbf{k}\in\mathbb{H}_{2,2,1},$ then
$N_{q}=\sqrt{1^{2}+2\cdot2^{2}+2\cdot1^{2}+1\cdot5^{2}}=\allowbreak6$ and we
can write

\begin{center}
$q=\dfrac{1}{6}+\dfrac{\sqrt{35}}{6}\dfrac{\left(  2,1,5\right)  }{\sqrt{35}%
}=\cos\theta+\dfrac{\left(  2,1,5\right)  }{\sqrt{35}}\sin\theta$
\end{center}

where $\boldsymbol{\varepsilon}_{0}=\dfrac{1}{\sqrt{35}}\left(  2,1,5\right)
$ is a unit vector in $%
\mathbb{R}
_{2,2,1}^{3}$ with $\boldsymbol{\varepsilon}_{0}^{2}=-1$.

\begin{theorem}
Each a unit elliptic quaternion represents an elliptical rotation on the an
ellipsoid. If
\[
q=q_{0}+q_{1}\mathbf{i}+q_{2}\mathbf{j}+q_{3}\mathbf{k}=\cos\theta
+\boldsymbol{\varepsilon}_{0}\sin\theta\in\mathbb{H}_{a_{1},a_{2},a_{3}}%
\]
is a unit elliptic quaternion, then the linear map $R_{\theta}\left(
\mathbf{v}\right)  =q\mathbf{v}q^{-1}$ gives an elliptical rotation through
the elliptical angle $2\theta,$ about the axis $\boldsymbol{\varepsilon}_{0},$
where $\mathbf{v}\in%
\mathbb{R}
^{3}.$ The elliptical rotation matrix to corresponding to the quaternion $q$
is
\begin{equation}
R_{\theta}^{q}=\left[
\begin{array}
[c]{ccc}%
q_{0}^{2}{\small +}a_{1}q_{1}^{2}-a_{2}q_{2}^{2}-a_{3}q_{3}^{2} & 2a_{2}%
q_{1}q_{2}-2\dfrac{q_{0}q_{3}\Delta}{a_{1}} & 2a_{3}q_{1}q_{3}+2\dfrac
{q_{0}q_{2}\Delta}{a_{1}}\\
2a_{1}q_{1}q_{2}+2\dfrac{q_{0}q_{3}\Delta}{a_{2}} & q_{0}^{2}-a_{1}q_{1}%
^{2}{\small +}a_{2}q_{2}^{2}-a_{3}q_{3}^{2} & 2a_{3}q_{2}q_{3}-2\dfrac
{q_{0}q_{1}\Delta}{a_{2}}\\
2a_{1}q_{1}q_{3}-2\dfrac{q_{0}q_{2}\Delta}{a_{3}} & 2a_{2}q_{2}q_{3}%
+2\dfrac{q_{0}q_{1}\Delta}{a_{3}} & q_{0}^{2}-a_{1}q_{1}^{2}-a_{2}q_{2}%
^{2}{\small +}a_{3}q_{3}^{2}%
\end{array}
\right]  . \label{quatrot}%
\end{equation}

\end{theorem}

\begin{proof}
It can be seen that $R_{\theta}$ is a linear transformation and preserves the
norm. Using the equalities,%
\begin{align*}
R_{\theta}\left(  \mathbf{i}\right)   &  =\left(  a_{1}q_{1}^{2}q_{0}%
^{2}{\small -}q_{2}^{2}a_{2}{\small -}q_{3}^{2}a_{3}\right)  \mathbf{i}%
+2\left(  a_{1}q_{1}q_{2}{\small +}q_{0}q_{3}\sqrt{\dfrac{a_{1}a_{3}}{a_{2}}%
}\right)  \mathbf{j}{\small +}2\left(  q_{1}q_{3}a_{1}{\small -}q_{0}%
q_{2}\sqrt{\dfrac{a_{1}a_{2}}{a_{3}}}\right)  \mathbf{k,}\\
R_{\theta}\left(  \mathbf{j}\right)   &  =2\left(  a_{2}q_{1}q_{2}%
{\small -}q_{0}q_{3}\sqrt{\dfrac{a_{2}a_{3}}{a_{1}}}\right)  \mathbf{i}%
{\small +}\left(  a_{2}q_{2}^{2}+q_{0}^{2}{\small -}q_{1}^{2}a_{1}%
{\small -}q_{3}^{2}a_{3}\right)  \mathbf{j}{\small +}2\left(  a_{2}q_{2}%
q_{3}{\small +}q_{0}q_{1}\sqrt{\dfrac{a_{1}a_{2}}{a_{3}}}\right)
\mathbf{k,}\\
R_{\theta}\left(  \mathbf{k}\right)   &  =2\left(  a_{3}q_{1}q_{3}%
{\small +}q_{0}q_{2}\sqrt{\dfrac{a_{2}a_{3}}{a_{1}}}\right)  \mathbf{i}%
{\small +}2\left(  a_{3}q_{2}q_{3}{\small -}q_{0}q_{1}\sqrt{\dfrac{a_{1}a_{3}%
}{a_{2}}}\right)  \mathbf{j}{\small +}\left(  a_{3}q_{3}^{2}+q_{0}%
^{2}{\small -}q_{1}^{2}a_{1}{\small -}q_{2}^{2}a_{2}\right)  \mathbf{k,}%
\end{align*}
we can obtain (\ref{quatrot}). So, the rotation matrix (\ref{quatrot}) is an
elliptical rotation matrix on the ellipsoid $a_{1}x^{2}+a_{2}y^{2}+a_{3}%
z^{2}=1$. That is, the equalities $\det R_{\theta}=1$ and $R_{\theta}%
^{t}\Omega R_{\theta}=\Omega$ are satisfied. Also, note that, if we take
$a_{1}=a_{2}=a_{3}=1$, the standard rotation matrix is obtained. Now, let's
choose an orthonormal set $\left\{  \boldsymbol{\varepsilon}_{0}%
,\boldsymbol{\varepsilon}_{1},\boldsymbol{\varepsilon}_{2}\right\}  $
satisfying the equalities
\[
\mathcal{V}\left(  \boldsymbol{\varepsilon}_{0}\times\boldsymbol{\varepsilon
}_{1}\right)  =\boldsymbol{\varepsilon}_{2},\ \ \mathcal{V}\left(
\boldsymbol{\varepsilon}_{2}\times\boldsymbol{\varepsilon}_{0}\right)
=\boldsymbol{\varepsilon}_{1},\ \ \mathcal{V}\left(  \boldsymbol{\varepsilon
}_{1}\times\boldsymbol{\varepsilon}_{2}\right)  =\boldsymbol{\varepsilon}%
_{0}.
\]
If $\boldsymbol{\varepsilon}$ is a vector in the plane of the
$\boldsymbol{\varepsilon}_{0}$ and $\boldsymbol{\varepsilon}_{1}$, we can
write it as
\[
\boldsymbol{\varepsilon}=\cos\alpha\boldsymbol{\varepsilon}_{0}+\sin
\alpha\boldsymbol{\varepsilon}_{1}.
\]
To compute $R_{\theta}^{q}\left(  \boldsymbol{\varepsilon}\right)
=q\boldsymbol{\varepsilon}q^{-1}$, let's find how $\boldsymbol{\varepsilon
}_{0}$ and $\boldsymbol{\varepsilon}_{1}$ change under the transformation
$R_{\theta}^{q}.$ Since $\mathbf{V}_{q}$ is parallel to
$\boldsymbol{\varepsilon}_{0},$ we have $q\boldsymbol{\varepsilon}%
_{0}=\boldsymbol{\varepsilon}_{0}q$ by (\ref{vect}) and $R_{q}\left(
\boldsymbol{\varepsilon}_{0}\right)  =q\boldsymbol{\varepsilon}_{0}%
q^{-1}=\boldsymbol{\varepsilon}_{0}qq^{-1}=\boldsymbol{\varepsilon}_{0}$. So,
$\boldsymbol{\varepsilon}_{0}$ is not changed under the transformation
$R_{\theta}^{q}$. It means that $\boldsymbol{\varepsilon}_{0}$ is the rotation
axis. On the other hand,%
\begin{align*}
R_{q}\left(  \boldsymbol{\varepsilon}_{1}\right)   &
=q\boldsymbol{\varepsilon}_{1}q^{-1}\\
&  =\left(  \cos\theta+\boldsymbol{\varepsilon}_{0}\sin\theta\right)
\boldsymbol{\varepsilon}_{1}\left(  \cos\theta-\boldsymbol{\varepsilon}%
_{0}\sin\theta\right) \\
&  =\boldsymbol{\varepsilon}_{1}\cos^{2}\theta-\cos\theta\sin\theta\left(
\boldsymbol{\varepsilon}_{1}\boldsymbol{\varepsilon}_{0}\right)  +\cos
\theta\sin\theta\left(  \boldsymbol{\varepsilon}_{0}\boldsymbol{\varepsilon
}_{1}\right)  -\left(  \boldsymbol{\varepsilon}_{0}\boldsymbol{\varepsilon
}_{1}\right)  \boldsymbol{\varepsilon}_{0}\sin^{2}\theta.
\end{align*}
Since we know that $\boldsymbol{\varepsilon}_{1}\boldsymbol{\varepsilon}%
_{0}=\mathcal{V}\left(  \boldsymbol{\varepsilon}_{1}\times
\boldsymbol{\varepsilon}_{0}\right)  =-\mathcal{V}\left(
\boldsymbol{\varepsilon}_{0}\times\boldsymbol{\varepsilon}_{1}\right)
=-\boldsymbol{\varepsilon}_{0}\boldsymbol{\varepsilon}_{1}%
=-\boldsymbol{\varepsilon}_{2}$ for orthogonal, pure quaternions, we obtain%
\begin{align*}
R_{q}\left(  \boldsymbol{\varepsilon}_{1}\right)   &  =\boldsymbol{\varepsilon
}_{1}\cos^{2}\theta+\left(  \boldsymbol{\varepsilon}_{1}%
\boldsymbol{\varepsilon}_{0}\right)  \boldsymbol{\varepsilon}_{0}\sin
^{2}\theta+2\boldsymbol{\varepsilon}_{2}\cos\theta\sin\theta\\
&  =\boldsymbol{\varepsilon}_{1}\cos^{2}\theta+\boldsymbol{\varepsilon}%
_{1}\boldsymbol{\varepsilon}_{0}^{2}\sin^{2}\theta+2\boldsymbol{\varepsilon
}_{2}\cos\theta\sin\theta\\
&  =\boldsymbol{\varepsilon}_{1}\cos2\theta+\boldsymbol{\varepsilon}_{2}%
\sin2\theta
\end{align*}
That is, $\boldsymbol{\varepsilon}$ is rotated through the elliptical angle
$2\theta$ about $\boldsymbol{\varepsilon}_{0}$ by the transformation
$R_{q}(\boldsymbol{\varepsilon}).$
\end{proof}

\setcounter{theorem}{0}

\begin{corollary}
All elliptical rotations on an ellipsoid can be represented by elliptic
quaternions which is defined for that ellipsoid.
\end{corollary}

\setcounter{theorem}{5}

\begin{example}
Let's find the general elliptical rotation matrix for the ellipsoid
$2x^{2}+2y^{2}+z^{2}=1.$ Using (\ref{quatrot})$,$ we obtain,%
\[
R_{\theta}^{q}=\left[
\begin{array}
[c]{ccc}%
q_{0}^{2}+2q_{1}^{2}-2q_{2}^{2}-q_{3}^{2} & 4q_{1}q_{2}-2q_{0}q_{3} &
2q_{0}q_{2}+2q_{1}q_{3}\\
2q_{0}q_{3}+4q_{1}q_{2} & q_{0}^{2}-2q_{1}^{2}+2q_{2}^{2}-q_{3}^{2} &
2q_{2}q_{3}-2q_{0}q_{1}\\
4q_{1}q_{3}-4q_{0}q_{2} & 4q_{0}q_{1}+4q_{2}q_{3} & q_{0}^{2}-2q_{1}%
^{2}-2q_{2}^{2}+q_{3}^{2}%
\end{array}
\right]  .
\]
Here, $\det R_{\theta}=\left(  q_{0}^{2}+2q_{1}^{2}+2q_{2}^{2}+q_{3}%
^{2}\right)  ^{3}=1$ and $R_{\theta}^{t}\Omega R=\Omega$ where $\Omega
=diag\left(  2,2,1\right)  .$ For example, the unit quaternion $q=\left(
0,1/2,1/2,0\right)  $ represents an elliptical rotation on the ellipsoid
$2x^{2}+2y^{2}+z^{2}=1$ through the elliptical angle $\pi,$ about the axis
$\left(  1/2,1/2,0\right)  .$ And the elliptical rotation matrix is
\end{example}

\begin{center}
$R_{\pi}^{q}=\left[
\begin{array}
[c]{ccc}%
0 & 1 & 0\\
1 & 0 & 0\\
0 & 0 & -1
\end{array}
\right]  .$
\end{center}

\subsection{An Algorithm}

Generating 3-dimensional rotation matrix that rotates $x=\left(  x_{1}%
,y_{1},y_{1}\right)  $ to $y=\left(  x_{2},y_{2},z_{2}\right)  $ elliptically
on the ellipsoid $a_{1}x^{2}+a_{2}y^{2}+a_{3}z^{2}=1$

\hrulefill

\textbf{Step 1.} Write, $\boldsymbol{a=}\left(  a_{1},a_{2},a_{3}\right)  $
for the given ellipsoid $a_{1}x^{2}+a_{2}y^{2}+a_{3}z^{2}=1$ where
$a_{1},a_{2},a_{3}\in%
\mathbb{R}
^{+}.$

\textbf{Step 2.} Define Scalar Product $\mathcal{B}$, norm of a vector and
Scalar product constant $\Delta$ as follows :%
\begin{align*}
\mathcal{B}\left(  \boldsymbol{x,y,a}\right)   &  =\mathbf{B}\left(
x_{1},y_{1},z_{1},x_{2},y_{2},z_{2},a_{1},a_{2},a_{3}\right)  =a_{1}x_{1}%
x_{2}+a_{2}y_{1}y_{2}+a_{3}z_{1}z_{2},\\
\mathcal{N}\left(  \mathbf{x,}\boldsymbol{a}\right)   &  =\sqrt{\mathcal{B}%
\left(  \boldsymbol{x,x,a}\right)  },\\
\Delta &  =\sqrt{a_{1}a_{2}a_{3}}%
\end{align*}
where $\boldsymbol{x}=\left(  x_{1},y_{1},y_{1}\right)  $ and $\boldsymbol{y}%
=\left(  x_{2},y_{2},z_{2}\right)  .$

\textbf{Step 3.} Define Vector Product $\mathcal{V}$ as%
\begin{align*}
\mathcal{V}\left(  \mathbf{x,y}\boldsymbol{,a}\right)   &  =\mathbf{V}\left(
x_{1},y_{1},z_{1},x_{2},y_{2},z_{2},a_{1},a_{2},a_{3}\right) \\
&  =\left(  \dfrac{\Delta}{a_{1}}\left(  y_{1}z_{2}-y_{2}z_{1}\right)
,\dfrac{\Delta}{a_{2}}\left(  -x_{1}z_{2}+x_{2}z_{1}\right)  ,\dfrac{\Delta
}{a_{3}}\left(  x_{1}y_{2}-x_{2}y_{1}\right)  \right)  .
\end{align*}

\textbf{Step 4.} Choose the vectors $\boldsymbol{x}=\left(  x_{1},y_{1}%
,y_{1}\right)  $ and $\boldsymbol{y}=\left(  x_{2},y_{2},z_{2}\right)  $ to
find the elliptical rotation matrix that rotates $\boldsymbol{x}$ to
$\boldsymbol{y}$ elliptically on the ellipsoid.

\textbf{Step 5.} Find, $\mathcal{V}\left(  \mathbf{x,y}\boldsymbol{,a}\right)
$ and norm of the vectors $\boldsymbol{x,y}$ and $\mathcal{V}\left(
\mathbf{x,y}\boldsymbol{,a}\right)  $. That is, find $\mathcal{N}\left(
\mathbf{x,}\boldsymbol{a}\right)  ,$ $\mathcal{N}\left(  \mathbf{y,}%
\boldsymbol{a}\right)  $ and $\mathcal{N}\left(  \mathcal{V}\left(
\mathbf{x,y}\boldsymbol{,a}\right)  \mathbf{,}\boldsymbol{a}\right)  .$

\textbf{Step 6.} Find the rotation axis $\mathbf{u=}\left(  u_{1},u_{2}%
,u_{3}\right)  $ where

\begin{center}
$u_{1}=\dfrac{\Delta\left(  y_{1}z_{2}-y_{2}z_{1}\right)  }{a_{1}%
\cdot\mathcal{N}\left(  \mathcal{V}\left(  \mathbf{x,y}\boldsymbol{,a}\right)
\mathbf{,}\boldsymbol{a}\right)  },$ \ $u_{2}=\dfrac{-\Delta\left(  x_{1}%
z_{2}+x_{2}z_{1}\right)  }{a_{2}\cdot\mathcal{N}\left(  \mathcal{V}\left(
\mathbf{x,y}\boldsymbol{,a}\right)  \mathbf{,}\boldsymbol{a}\right)  },$
\ $u_{1}=\dfrac{\Delta\left(  x_{1}y_{2}-x_{2}y_{1}\right)  }{a_{3}%
\cdot\mathcal{N}\left(  \mathcal{V}\left(  \mathbf{x,y}\boldsymbol{,a}\right)
\mathbf{,}\boldsymbol{a}\right)  }.$
\end{center}

\textbf{Step 7.} Find the elliptical rotation angle using

\begin{center}
$\cos\theta=\dfrac{\mathcal{B}\left(  \boldsymbol{x,y,a}\right)  }%
{\sqrt{\mathcal{N}\left(  \mathbf{x,}\boldsymbol{a}\right)  }\sqrt
{\mathcal{N}\left(  \mathbf{y,}\boldsymbol{a}\right)  }}$
\end{center}

and define $C=\cos\theta$ and $S=\sin\theta$ where $S=\sqrt{1-C^{2}}.$

\textbf{Step 8. }Find the elliptical rotation matrix that rotates
$\boldsymbol{x}=\left(  x_{1},y_{1},y_{1}\right)  $ to $\boldsymbol{y}=\left(
x_{2},y_{2},z_{2}\right)  $ elliptically on the ellipsoid $a_{1}x^{2}%
+a_{2}y^{2}+a_{3}z^{2}=1$ using Rodrigues matrix%
\[
R\left(  \boldsymbol{u},C,S,\Delta\right)  =\left[
\begin{array}
[c]{ccc}%
a_{1}u_{1}^{2}{\small +}\left(  1{\small -}a_{1}u_{1}^{2}\right)  C &
{\small -}\dfrac{\Delta u_{3}S}{a_{1}}{\small -}a_{2}u_{1}u_{2}\left(
C{\small -}1\right)  & \dfrac{\Delta u_{2}S}{a_{1}}{\small -}a_{3}u_{1}%
u_{3}\left(  C{\small -}1\right) \\
\dfrac{\Delta u_{3}S}{a_{2}}{\small -}a_{1}u_{1}u_{2}\left(  C{\small -}%
1\right)  & a_{2}u_{2}^{2}{\small +}\left(  1{\small -}a_{2}u_{2}^{2}\right)
C & {\small -}\dfrac{\Delta u_{1}S}{a_{2}}{\small -}a_{3}u_{2}u_{3}\left(
C{\small -}1\right) \\
{\small -}\dfrac{\Delta u_{2}S}{a_{3}}{\small -}a_{1}u_{1}u_{3}\left(
C{\small -}1\right)  & \dfrac{\Delta u_{1}S}{a_{3}}{\small -}a_{2}u_{2}%
u_{3}\left(  C{\small -}1\right)  & a_{3}u_{3}^{2}{\small +}\left(
1{\small -}a_{3}u_{3}^{2}\right)  C
\end{array}
\right]
\]
where $\boldsymbol{u}=\left(  u_{1},u_{2},u_{3}\right)  .$

\textbf{Step 9.} Define the matrix

\begin{center}
$\mathcal{H}\left(  \mathbf{v,}\boldsymbol{a}\right)  =\dfrac{1}%
{\mathcal{N}\left(  \mathbf{v,}\boldsymbol{a}\right)  }\left[
\begin{array}
[c]{ccc}%
\mathcal{N}\left(  \mathbf{v,}\boldsymbol{a}\right)  -2a_{1}v_{1}^{2} &
-2a_{2}v_{1}v_{2} & -2a_{3}v_{1}v_{3}\\
-2a_{1}v_{2}v_{1} & \mathcal{N}\left(  \mathbf{v,}\boldsymbol{a}\right)
-2a_{2}v_{2}^{2} & -2a_{3}v_{2}v_{3}\\
-2a_{1}v_{3}v_{1} & -2a_{2}v_{3}v_{2} & \mathcal{N}\left(  \mathbf{v,}%
\boldsymbol{a}\right)  -2a_{3}v_{3}^{2}%
\end{array}
\right]  $
\end{center}

for a given $\boldsymbol{v=}\left(  v_{1},v_{2},v_{3}\right)  .$

\textbf{Step 10.} Find the elliptical rotation matrix that rotates
$\boldsymbol{x}=\left(  x_{1},y_{1},y_{1}\right)  $ to $\boldsymbol{y}=\left(
x_{2},y_{2},z_{2}\right)  $ elliptically on the ellipsoid $a_{1}x^{2}%
+a_{2}y^{2}+a_{3}z^{2}=1$ using Householder matrices

\begin{center}
$R\left(  \boldsymbol{x,y,a}\right)  =\mathcal{H}\left(  \boldsymbol{y}%
\mathbf{,}\boldsymbol{a}\right)  \mathcal{H}\left(  \boldsymbol{x+y}%
\mathbf{,}\boldsymbol{a}\right)  .$
\end{center}

\textbf{Step 11.} Define the set of Elliptic Quaternions $\mathbb{H}%
_{a_{1},a_{2},a_{3}}=\{q=q_{0}+q_{1}\mathbf{i}+q_{2}\mathbf{j}+q_{3}%
\mathbf{k,}$ $q_{0},q_{1},q_{2},q_{3}\in%
\mathbb{R}
\}$ with

\begin{center}
$\mathbf{i}^{2}=-a_{1},$ $\ \ \ \ \mathbf{j}^{2}=-a_{2},$ $\ \ \ \ \mathbf{k}%
^{2}=-a_{3}$
\end{center}

and

\begin{center}
$\mathbf{ij}=\dfrac{\Delta}{a_{3}}\mathbf{k=-ji},$ $\ \ \ \ \ \mathbf{jk}%
=\dfrac{\Delta}{a_{1}}\mathbf{i=-kj},$ $\ \ \ \ \ \ \ \mathbf{ki}%
=\dfrac{\Delta}{a_{2}}\mathbf{j=-ik.}$
\end{center}

\textbf{Step 12. }Find $c=\cos\dfrac{\theta}{2}=\sqrt{\dfrac{\cos\theta+1}{2}%
}$ and $s=\sqrt{1-c^{2}}.$ Define the quaternion
\[
q=\cos\dfrac{\theta}{2}+\boldsymbol{u}\sin\dfrac{\theta}{2}=c+su_{1}%
\mathbf{i}+su_{2}\mathbf{j}+su_{3}\mathbf{k}%
\]
where $\theta$ is the elliptical rotation angle and $u=\left(  u_{1}%
,u_{2},u_{3}\right)  =u_{1}i+u_{2}j+u_{3}k$ is the rotation axis obtained in
Step 6 and Step7.

\textbf{Step 13. }Find the elliptical rotation matrix that rotates
$\boldsymbol{x}$ to $\boldsymbol{y}$ elliptically on the ellipsoid using the matrix

\begin{center}
$R\left(  q,\boldsymbol{a},\Delta\right)  =\left[
\begin{array}
[c]{ccc}%
q_{0}^{2}{\small +}a_{1}q_{1}^{2}-a_{2}q_{2}^{2}-a_{3}q_{3}^{2} & 2a_{2}%
q_{1}q_{2}-2\dfrac{q_{0}q_{3}\Delta}{a_{1}} & 2a_{3}q_{1}q_{3}+2\dfrac
{q_{0}q_{2}\Delta}{a_{1}}\\
2a_{1}q_{1}q_{2}+2\dfrac{q_{0}q_{3}\Delta}{a_{2}} & q_{0}^{2}-a_{1}q_{1}%
^{2}{\small +}a_{2}q_{2}^{2}-a_{3}q_{3}^{2} & 2a_{3}q_{2}q_{3}-2\dfrac
{q_{0}q_{1}\Delta}{a_{2}}\\
2a_{1}q_{1}q_{3}-2\dfrac{q_{0}q_{2}\Delta}{a_{3}} & 2a_{2}q_{2}q_{3}%
+2\dfrac{q_{0}q_{1}\Delta}{a_{3}} & q_{0}^{2}-a_{1}q_{1}^{2}-a_{2}q_{2}%
^{2}{\small +}a_{3}q_{3}^{2}%
\end{array}
\right]  $
\end{center}

corresponding to $q=q_{0}+q_{1}\mathbf{i}+q_{2}\mathbf{j}+q_{3}\mathbf{k}.$

\textbf{Remark}

In the $n>3$ dimensional spaces, rotations can be classified such as simple,
composite, and isoclinic, depending on plane of rotation. Simple rotation is a
rotation with only one plane of rotation. In a simple rotation, there is a
fixed plane. The rotation is said to take place about this plane. So points do
not change their distance from this plane as they rotate. Orthogonal to this
fixed plane is the plane of rotation. The rotation is said to take place in
this plane. On the other hand, a rotation with two or more planes of rotation
is called a composite rotation. The rotation can take place in each plane of
rotation. These planes are orthogonal. In $%
\mathbb{R}
^{4}$ it is called a double rotation. A double rotation has two angles of
rotation, one for each plane of rotation. The rotation is specified by giving
the two planes and two non-zero angles $\beta$ and $\theta$ (if either angle
is zero, then the rotation is simple). Finally, the isoclinic rotation is a
special case of the composite rotation, when the two angles are equal
\cite{rot types}. In the $4$ dimensional Euclidean and Lorentzian spaces, a
skew symmetric matrix is decomposed as $A=\theta_{1}A_{1}+\theta_{2}A_{2}$
using two skew-symmetric matrices $A_{1}$ and $A_{2}$ satisfying the
properties $A_{1}A_{2}=0,$ $A_{1}^{3}=-A_{1}$ and $A_{2}^{3}=-A_{2}.$ Hence,
the Rodrigues and Cayley rotation formulas can be used to generate 4
dimensional rotation matrices (\cite{eberly}, \cite{J. Gallier1},
\cite{melek2}, \cite{caygallier}, and \cite{rot3}).


\begin{thebibliography}{99}                                                                                               %


\bibitem {Macykey1}D. S. Mackey, N. Mackey, F. Tisseur, G-reflectors :
Analogues of Householder transformations in scalar product spaces, Linear
Algebra and its Applications Vol. 385, (2004), 187--213.

\bibitem {Aragon1}M.A. Rodr\'{\i}guez-Andrade, , G. Arag\'{o}n-Gonz\'{a}lez,
J.L. Arag\'{o}n, L. Verde-Star, An algorithm for the Cartan-Dieudonn\'{e}
theorem on generalized scalar product spaces, Linear Algebra and Its
Applications, Vol. 434, Issue 5, (2011), 1238-1254.

\bibitem {J. Gallier1}J. H. Gallier : Geometric Methods and Applications, For
Computer Science and Engineering. Texts in Applied Mathematics \textbf{38
}(2011): 680p.

\bibitem {caygallier}Jean Gallier, Remarks on the Cayley Representation of
Orthogonal Matrices and Perturbing the Diagonal of a Matrix to Make it
Invertible, arXiv:math/0606320v2, (2013).

\bibitem {Gallier book}J. H. Gallier : Notes on Differential Geometry and Lie
Groups\textit{. }University of Pennsylvania (2014): 730p.

\bibitem {gallier 2}J. Gallier and D. Xu : Computing Exponentials of Skew
Symmetric Matrices and Logarithms of Orthogonal Matrices. International
Journal of Robotics and Automation \textbf{18 }(2000): 10-20.

\bibitem {affine}A. Revent\'{o}s Tarrida, Affine Maps, Euclidean Motions and
Quadrics, Srpinger, (2011), 411p.

\bibitem {qua2}J. Schmidt, H. Nieman : Using Quaternions for Parametrizing 3-D
Rotations in Unconstrained Nonlinear Optimization, Vision Modeling and
Visualization. Stuttgart, Germany (2001): 399-406.

\bibitem {cay3}A. N. Norris : Euler-Rodrigues and Cayley Formulae for Rotation
of Elasticity Tensors.\textit{ }Mathematics and Mechanics of Solids \textbf{13
}(2008): 465-498.

\bibitem {Cl }C.L. Bottasso and M. Barri : Integrating finite rotations.
Computer Methods in Applied Mechanics and Engineering \textbf{164 }(1998): 307-331.

\bibitem {qua3}L. Vicci : Quaternions and Rotations in 3-Space, The Algebra
and Geometric Interpretation. Microelectric Systems Laboratory, Department of
Computer Sciences, UNC (2001): TR01-014.

\bibitem {EC}E. Celledoni and N. Safstr\"{o}m: A Hamiltonian and
Multi-Hamiltonian formulation of a rod model using quaternions. Computer
Methods in Applied Mechanics and Engineering \textbf{199 }(2010)\textbf{:} 2813-2819.

\bibitem {Fuller}C. Fuller, A constructive proof of the
Cartan--Dieudonn\'{e}--Scherk Theorem in the real or complex case, Journal of
Pure and Applied Algebra 215 (2011) 1116--1126.

\bibitem {Uhlig}F. Uhlig, Constructive ways for generating (generalized) real
orthogonal matrices as products of (generalized) symmetries, Linear Algebra
Appl. 332--334 (2001) 459--467.

\bibitem {rot3}M. \"{O}zdemir, M. Erdo\u{g}du : On the Rotation Matrix in
Minkowski Space-time. Reports on Mathematical Physics \textbf{74} (2014): 27-38.

\bibitem {qua4}M. \"{O}zdemir, A.A. Ergin : Rotations with unit timelike
quaternions in Minkowski 3-space. Journal of Geometry and Physics \textbf{56
}(2006): 322-336.

\bibitem {qua5}M. \"{O}zdemir, M. Erdo\u{g}du, H. \c{S}im\c{s}ek : On the
Eigenvalues and Eigenvectors of a Lorentzian Rotation Matrix by Using Split
Quaternions.\textit{ }Adv. Appl. Clifford Algebras \textbf{24 }(2014): 179-192.

\bibitem {melek2}M. Erdo\u{g}du, M. \"{O}zdemir, Cayley Formula on Matrix in
Minkowski Space-time, Int. J. of Geometric Methods in Modern Physics (Accepted).

\bibitem {rot1}A. Jadczyk, J. Szulga : A Comment on "On the Rotation Matrix in
Minkowski Space-time" by \"{O}zdemir and Erdo\u{g}du. Reports on Mathematical
Physics \textbf{74 }(2014): 39-44.

\bibitem {Aragon2}G. Arag\'{o}n-Gonz\'{a}lez, J.L. Arag\'{o}n, M.A.
Rodr\'{\i}guez-Andrade, The decomposition of an orthogonal transformation as a
product of reflections, J. Math. Phys. 47 (2006), Art. No. 013509.

\bibitem {MEBIUS2}J. E. Mebius : Derivation of Euler-Rodrigues Formula for
three-dimensional rotations from the general formula for four dimensional
rotations, arxiv: math.GM (2007).

\bibitem {qua6}J. E. Mebius : A Matrix Based Proof of the Quaternion
Representation Theorem for Four Dimensional Rotations (2005) arxiv:
math.GM/0501249v1{\small .}

\bibitem {serre}D. Serre : Matrices: Theory and Applications, Graduate text in
Mathematics, Springer - Verlag, London (2002).

\bibitem {rot2}B. B\"{u}k\c{c}\"{u} : On the Rotation Matrices in
Semi-Euclidean Space\textit{.} Commun. Fac. Sci. Univ. Ank. Series A1.
\textbf{55 }(2006): 7-13.

\bibitem {qua7}J. L. Weiner and G.R. Wilkens : Quaternions and Rotations
in\textit{ }$\mathbb{E}^{4}.$\ The American Mathematical Monthly (2005): 69-76.

\bibitem {rod1}R. W. Brackett : Robotic Manipulators and the Product of
Exponentials Formula. Mathematical Theory and Networks and Systems. Proceeding
of International Symposium. Berlin, Springer-Verlag (1984): 120-127.

\bibitem {rod2}M. R. Murray, Z. Li, S.S. Sastry : A Mathematical Introduction
to Robotic Manipulation. Boca Raton F.L. CRC Press (1994).

\bibitem {rod3}L. Kula, M.K. Karacan, Y. Yayl\i\ : Formulas for the
Exponential of Semi Symmetric Matrix of order $4.$ Mathematical and
Computational Applications \textbf{10 }(2005): 99-104.

\bibitem {rod4}T. Politi : A Formula for the Exponential of a Real
Skew-Symmetric Matrix of Order 4.\textit{ }BIT Numerical Mathematics
\textbf{41} (2001): 842-845.

\bibitem {cay1}J. M. Selig : Cayley Maps for SE(3). 12th IFToMM World
Congress, Besancon (2007): 18-21.

\bibitem {cay2}A. Cayley : Sur Quelques Proprietes des Determinants Gauches.
The Collected Papers of Arthur Cayley SC.D.F.R.S. Cambridge University Press (1889).

\bibitem {cay4}S. \"{O}zkald\i,\ H. G\"{u}ndo\u{g}an : Cayley Formula, Euler
Parameters and Rotations in 3- Dimensional Lorentzian Space. Advances in
Applied Clifford Algebras \textbf{20 }(2010): 367-377.

\bibitem {eberly}D. Eberly : Constructing Rotation Matrices Using Power
Series. Geometric Tools LLC (2007): http://geometrictools.com/.

\bibitem {rot types}L. Pertti : Clifford algebras and spinors. Cambridge
University Press (2001): ISBN:978-0-521-00551-7.

\bibitem {yuzyil}Barry A. Cipra, The Best of the 20th Century: Editors Name
Top 10 Algorithms, SIAM News, Vol. 33, No. 4, (2000).
\end{thebibliography}
\end{document}